\documentclass{amsart}
\usepackage{amsfonts,amssymb,amsmath, dsfont}
\usepackage{upgreek}
\newtheorem{theorem}{Theorem}[section]
\newtheorem{proposition}[theorem]{Proposition}
\newtheorem{lemma}[theorem]{Lemma}
\newtheorem{corollary}[theorem]{Corollary}
\theoremstyle{remark}
\newtheorem{rem}[theorem]{Remark}

\def \IP {\mathbb{P}}

\def \N {{\mathbb N}}
\def \P {{\mathbb P}}

\def \R {{\mathbb R}}
\def \Z {{\mathbb Z}}

\def \ba {\begin{array}}
\def \ea {\end{array}}

\def \cE {{\mathcal E}}
\def \cF {{\mathcal F}}
\def \cG {{\mathcal G}}

\def \cI {{\mathcal I}}
\def \cL {{\mathcal L}}

\def\one{\rlap{\mbox{\small\rm 1}}\kern.15em 1}

\def \un {\underline}
\def \ov {\overline}
\def \td {\tilde}
\def \Ll {\left}
\def \Rr {\right}
\def \subset {\subseteq}
\def \d {\mathrm d}
\def \tdj {\td{\jmath}}
\def \emptyset {\varnothing}

\begin{document}

\title[Exponential extinction time of the contact process]{Exponential extinction time of the contact process on finite graphs}

\author[T.\ Mountford, J.-C.\ Mourrat, D.\ Valesin, Q.\ Yao]{Thomas Mountford\textsuperscript{1}, Jean-Christophe Mourrat\textsuperscript{1}, Daniel Valesin\textsuperscript{2,3}
and Qiang Yao\textsuperscript{4}}
\footnotetext[1]{\'Ecole Polytechnique F\'ed\'erale de Lausanne,
D\'epartement de Math\'ematiques,
1015 Lausanne, Switzerland}
\footnotetext[2]{University of British Columbia, Department of Mathematics, V6T1Z2 Vancouver, Canada}
\footnotetext[3]{Research funded by the Post-Doctoral Research Fellowship of the Government of Canada}
\footnotetext[4]{Department of Statistics and Actuarial Science,
East China Normal University,
Shanghai 200241, China}

\begin{abstract}
We study the extinction time $\uptau$ of the contact process on finite trees of bounded degree. We show that, if the infection rate is larger than the critical rate for the contact process on $\Z$, then, uniformly over all trees of degree bounded by a given number, the expectation of $\uptau$ grows exponentially with the number of vertices. Additionally, for any sequence of growing trees of bounded degree, $\uptau$ divided by its expectation converges in distribution to the unitary exponential distribution. These also hold if one considers a sequence of graphs having spanning trees with uniformly bounded degree. Using these results, we consider the contact process on a random graph with vertex degrees following a power law.
%, as proposed by Newman-Strogatz-Watts in \cite{NSW}
Improving a result of Chatterjee and Durrett \cite{CD}, we show that, for any infection rate, the extinction time for the contact process on this graph grows exponentially with the number of vertices.

\bigskip


\medskip


\end{abstract}

%\noindent\textit{Mathematics Subject Classification.}  Primary: 82C22. Secondary: 05C80.  \\\textit{Key words:}
%  contact process, interacting particle systems, metastability.

\maketitle
%
%
%
%
%%%%%%%%%%%%%%%%%%%%%%%%%%%%%%%%%%%%%%%%%%%%%%%%%%%%%%%%%%%%%%
%%%%%%%%%%%%%%%%%%%%%%%%%%%%%%%%%%%%%%%%%%%%%%%%%%%%%%%%%%%%%%
%
%
%
\section{Introduction}
\label{s:intro}
\setcounter{equation}{0}
The contact process with parameter $\lambda > 0$ on a graph $G = (V, E)$ is a continuous-time Markov process $(\xi_t)_{t \geq 0}$ with state space $\{0,1\}^V$ and generator
\begin{equation}  \label{eq1gen}
\Omega f(\xi) = \sum_{x \in V} \left(f(\phi_x\xi) - f(\xi) \right) + \lambda \cdot \sum_{e\in E}  \left(f(\phi_e\xi) - f(\xi) \right),
\end{equation}
where $f$ is any local function on $\{0,1\}^V$ and, given $x \in V$ and $\{y, z\} \in E,$ we define $\phi_x\xi,\; \phi_{\{y,z\}}\xi \in \{0,1\}^V$ by
$$\phi_x\xi(w) = \left|\begin{array}{ll}0 &\text{if } w = x;\\\xi(w)&\text{otherwise;} \end{array}\right.\qquad \phi_{\{y,z\}}\xi(w) = \left|\begin{array}{ll}\max(\xi(y),\xi(z))&\text{if } w \in \{y,z\};\\\xi(w) &\text{otherwise.} \end{array} \right.$$
Given $A \subset V$, we write $(\xi^A_t)_{t \ge 0}$ to denote the contact process started from the initial configuration that is equal to 1 at vertices of $A$ and 0 at other vertices. When we write $(\xi_t)$, with no superscript, the initial configuration will either be clear from the context or unimportant. We often abuse notation and associate configurations $\xi \in \{0,1\}^V$ with the corresponding sets $\{x\in V: \xi(x) = 1\}$.

The contact process is a model for the spread of an infection in a population. Vertices of the graph (sometimes referred to as \textit{sites}) represent individuals. In a configuration $\xi \in \{0,1\}^V$, individuals in state 1 are said to be \textit{infected}, and individuals in state 0 are \textit{healthy}. Pairs of individuals that are connected by edges in the graph are in proximity to each other in the population. The generator (\ref{eq1gen}) gives two types of transition for the dynamics. First, infected individuals \textit{heal} with rate 1. Second, given two individuals in proximity so that one is infected and the other is not, with rate $\lambda$ there occurs a \textit{transmission}, as a consequence of which both individuals end up infected.

The configuration $\underline{0} \in\{0,1\}^V$ that is equal to zero at all vertices is a trap for $(\xi_t)$. For certain choices of the underlying graph $G$ and the parameter $\lambda$, it may be the case that the probability of the event $\{\underline{0} \text{ is never reached}\}$ is positive even if the process starts from finitely many infected sites. In fact, whether or not this probability is positive does not depend on the set of initially infected sites, as long as this set is nonempty and finite. We say that the process \textit{survives} if this probability is positive; otherwise we say that the process \textit{dies out}.

In order to be able to motivate and state our results, we will now list some of the properties of the contact process for certain choices of the graph $G$, namely: the lattice $\mathbb{Z}^d$, $d$-regular infinite trees and the finite counterparts of these graphs. For proofs of these properties and a detailed treatment of the topic, we refer the reader to \cite{lig85,lig99}.

Let us start with $\Z^d$, the $d$-dimensional integer lattice endowed with edge set $\{\{x,y\}: \|x-y\| = 1\}$, where $\|\cdot\|$ denotes Euclidean distance. In this case, there exists a number $\lambda_c = \lambda_c(\Z^d)$ such that, depending on whether $\lambda < \lambda_c,\;\lambda = \lambda_c$ or $\lambda > \lambda_c$, the process exhibits different behavior; these three regimes are respectively called \textit{subcritical}, \textit{critical} and \textit{supercritical}. It is known that the process survives if and only if it is supercritical. In this case, the process is known to \textit{survive strongly}, meaning that (for any nonempty initial configuration) with positive probability, each site becomes infected at arbitrarily large times:
$$\lambda > \lambda_c \Longrightarrow P\left[\forall x, \forall t, \exists t'> t: \xi_{t'}(x) = 1\right] > 0.$$

The interest in the contact process on trees was prompted after it was discovered in \cite{P} that death and strong survival are not the only possibilities in this case. For $d \geq 2$, let $T_d$ denote the infinite $(d+1)$-regular tree with a distinguished vertex $o$ called the root. The different phases of the process are captured by two constants $\lambda_1(T_d) < \lambda_2(T_d)$. If $\lambda \leq \lambda_1$, $(\xi^{\{o\}}_t)$ dies out, and if $\lambda > \lambda_2$, it survives strongly. If $\lambda \in (\lambda_1,\;\lambda_2]$, then the process \textit{survives weakly}, meaning that it survives but does not survive strongly. This implies that, even though the infection has positive probability of always being present on the graph, each individual site eventually becomes permanently healthy.

If $G$ is a finite graph, the contact process on $G$ dies out. Given $A\subset V$, define $\uptau_G^A = \inf\{t: \xi^A_t = \underline{0}\}$, the \textit{extinction time} for the process started from occupancy in $A$. We may omit the subscript $G$ when the context is clear enough, and simply write $\uptau$ when the contact process is started from full occupancy, that is, $\uptau = \uptau^{\un{1}}$. The distribution of $\uptau$ and the behavior of the process until this time can be very interesting. Consider the graph $\{0,\ldots,n\}^d$ (viewed as a subgraph of $\Z^d$) and the distribution of $\uptau$ for this graph, as $n$ goes to infinity. The three regimes of the infinite-volume process manifest themselves in the following way. If $\lambda < \lambda_c(\Z^d)$, then $\uptau/\log n$ converges in probability to a constant \cite{durliu}. If $\lambda = \lambda_c$, then $\uptau/n \to \infty$ and $\uptau/n^4 \to 0$ in probability \cite{dursctan}. If $\lambda > \lambda_c$, then $\lim_{n\to\infty} \log E[\uptau]/n^d$ exists and $\uptau/E[\uptau]$ converges in distribution to the unit exponential distribution \cite{dursc,tommeta,tomexp}. In the latter case, the process is said to exhibit \textit{metastability}, meaning that it persists for a long time in a state that resembles an equilibrium and then quickly moves to its true equilibrium ($\underline{0}$ in this case). Metastability for the contact process in this setting was also studied in  \cite{eulalia} and \cite{schonmeta}.

For the case of finite trees, the picture is less complete, and the available results concerning the extinction time are contained in \cite{St}. Fix $d \geq 2$, let $T_d^h$ be the finite subgraph of $T_d$ defined by considering up to $h$ generations from the root and again take the contact process started from full occupancy on this graph, with associated extinction time $\uptau$. If $\lambda < \lambda_2$, then there exist constants $c, C>0$ such that $P(ch \leq \uptau \leq Ch) \to 1$ as $h \to \infty$. If $\lambda > \lambda_2$, then for any $\sigma < 1$ there exist $c_1, c_2 >0$ such that
$$P\left[\uptau > c_1 e^{c_2(\sigma d)^h}\right] \to 1 \text{ as } h \to \infty.$$
Notice that the above implies that $\uptau$ is at least as large as a stretched exponential function of the number of vertices, $(d+1)^h$. As far as we know, no results are available concerning finite graphs that are not regular.

\medskip

For $n \in \N$ and $d>0$, let $\Lambda(n,d)$ be the set of all trees with $n$ vertices and degree bounded by $d$, and let $\cG(n,d)$ be the set of graphs having a spanning tree in $\Lambda(n,d)$. In this paper, we prove the following theorems.
\begin{theorem}
\label{thm1main1}
For any $d \geq 2$ and $\lambda > \lambda_c(\Z)$, there exists $c>0$ such that, for any~$n$ large enough,
$$\inf_{T \in \Lambda(n,d)} \frac{\log E[\uptau_T]}{n} \geq c.$$
\end{theorem}

\begin{theorem}
\label{thm2main2}
Let $d \geq 2$, $\lambda > \lambda_c(\Z)$, and $G_n \in \cG(n,d)$. The distribution of $\uptau_{G_n}/E[\uptau_{G_n}]$ converges to the unitary exponential distribution as $n$ tends to infinity.
\end{theorem}

\begin{theorem}
\label{thm3main3}
Let $d \geq 2$ and $\lambda> \lambda_c(\Z)$. There exists $c >0$ such that
$$\inf_{T \in \Lambda(n,d)} P\left[\uptau_T \ge e^{cn}\right] \to 1 \text{ as } n \to \infty.$$
\end{theorem}

By attractiveness, one can replace $\Lambda(n,d)$ by the set of all graphs having a subgraph in $\Lambda(n,d)$ in Theorems~\ref{thm1main1} and \ref{thm3main3}, and in particular, one can replace $\Lambda(n,d)$ by $\cG(n,d)$. For instance, the above results cover the case of any sequence of increasingly large connected subsets of $\Z^d$. At the cost of requiring $\lambda > \lambda_c(\Z)$, we thus recover and extend previously mentionned results, without any strong assumption on the regularity of the graph. For these values of $\lambda$, this shows in particular that on regular trees with finite depth, the extinction time is not only larger than a stretched exponential function of the number of vertices, but actually an exponential function.

In order to exemplify further the usefulness of our results, we then consider the contact process on Newman-Strogatz-Watts (NSW) random graphs, as considered in \cite{NSW} and \cite{CD}. Let us define them. For any $n \in \N$, we construct a graph $G^n$ on $n$ vertices. The vertex set is simply $\{1, \ldots, n\}$. The random set of edges will be constructed from a probability $p$ on $\{3,4,\ldots\}$ with the property that, for some $a>1$, $c_0 = \lim_{m\to\infty} p(m)/m^a$ exists and is in $(0, \infty)$. We let $d_1, \cdots, d_n$ be independent random variables distributed according to $p$, and conditioned on the event that $d_1 + \cdots + d_n$ is even. Next, from each vertex $i \in \{1, \ldots, n\}$ we place $d_i$ \textit{half-edges}; when two half-edges are connected, an edge is formed. We pair up the $d_1+\cdots+d_n$ half-edges in a random way that is uniformly chosen among all possibilities. Note that this can produce multiple edges between two vertices and also loops (edges that start and finish at the same vertex). We then take the contact process with parameter $\lambda > 0$ on this random graph. Notice that the generator given by (\ref{eq1gen}) does not exclude the case of multiple edges or loops: the latter have no effect in the dynamics and the former increase the rate of transmission between vertices.

Let us write $\mathbb{P}$ to denote a probability measure under which both the random graph and the contact process on this graph are defined. In \cite{CD}, it is shown that, for any $\lambda > 0$ and any $\delta > 0$, we have $\IP[\uptau(G^n) \ge e^{n^{1-\delta}}] \to 1$ as $n \to \infty$. We improve this and show that
\begin{theorem}
\label{thm1cd1}
For any $\lambda > 0$, there exists $c > 0$ such that
$$\mathbb{P}\left[\uptau_{G^n} \ge e^{cn} \right] \to 1 \text{ as } n \to \infty.$$
\end{theorem}

Although it would be simple to deduce Theorem~\ref{thm1cd1} from Theorem~\ref{thm3main3} assuming $\lambda > \lambda_c(\Z)$, we stress that here we cover any non-zero infection parameter. Theorem \ref{thm1cd1} is true for all $a>2$, but we only give the proof for $a > 3$, which is the harder case (when we increase $a$, the degrees of the vertices become stochastically smaller, so the graph is less connected and the contact process survives for a shorter time). Nevertheless in the case $a>3$ we have the advantage that the law $p$ has finite second moment.

\medskip

Let us briefly explain the proofs of our results and how they are organized in the paper. Section~\ref{s:remind} is a brief reminder on some properties of the contact process that will be useful for our purposes. In Section \ref{s:metastab}, we show a weaker version of Theorem~\ref{thm1main1}, which states that the expectation of the extinction time is larger than $e^{cn^{\alpha}}$ for some $\alpha > 0$. In order to do this, we consider two cases: either the tree contains a large segment, or it contains a large number of disjoint smaller segments. In the first case, the result follows from the known behavior of the extinction time on finite intervals of $\Z$. In the second case, we adapt an argument of \cite{CD} and show that, even if the segments are not too large, the time scale of extinction in individual segments is large enough for the infection to spread to other, possibly inactive, segments, so that the segments can jointly sustain activity for the desired amount of time. At this point, using a general metastability argument from \cite{tommeta}, we prove Theorem \ref{thm2main2}.

Given a tree $T \in \Lambda(n,d)$, we decompose it into two subtrees $T_1, T_2$ by removing an edge; we argue that this can be done so that $T_1$ and $T_2$ both contain a non-vanishing proportion of the vertices of $T$. In Section \ref{s:comparison2}, we bound the contact process $(\xi_t)_{t \ge 0}$ on $T$ from below by a pair of processes $(\zeta_{T_1,t})_{t \ge 0}$ on $T_1$ and $(\zeta_{T_2,t})_{t \ge 0}$ on $T_2$. The process $\zeta_{T_1}$ evolves as a contact process on $T_1$ until extinction. However, once extinct, the process stays extinct for some time, and then, as the Phoenix, it rises back from the ashes. This rebirth of the process reflects the fact that, as long as the true process $\xi$ has not died out, the tree $T_1$ constantly receives new infections that can restore its activity. The process $\zeta_{T_2}$ evolves independently, following the same rules. We show that the true process $\xi$ dominates $\zeta_{T_1} \cup \zeta_{T_2}$ up to the extinction of $\xi$, with probability close to $1$. With this comparison at hand, we argue that, modulo a factor that is polynomial in the number of vertices, the expected extinction time for $T$ is larger than the product of the expected extinction times for $T_1$ and $T_2$. This, together with the lower bound $e^{cn^\alpha}$ mentioned in the last paragraph, is then used to prove Theorem~\ref{thm1main1}, from which Theorem~\ref{thm3main3} follows.

In Section \ref{s:discrete}, we re-state some of the results explained above for a discrete-time version of the contact process.

In Section \ref{s:nsw}, we turn to the NSW random graph $G^n$. We present an algorithm that finds with high probability a certain subgraph $G'$ of $G^n$ containing a large quantity of vertices with degree above a certain threshold $M$ ($M$ depends on $\lambda$ but not on $n$). The algorithm also guarantees that most of these vertices are not isolated from other vertices with degree above $M$. Next, a tree $T$ and a mapping $\theta$ from the vertices of $T$ to those of $G'$ are given. $\theta$ has the properties that, for any $x$, $\theta(x)$ has degree larger than $M$ and, if $x, y$ are neighbors in $T$, then $\theta(x)$ and $\theta(y)$ are not far from each other in $G'$. By considering $(\xi_t \cap G')$ only at values of $t$ that are integer multiples of a large constant $R$, we then define a discrete-time version of the contact process on $T$, denoted $(\eta_k)_{k \geq 1}$. The construction is such that, if a vertex $\theta(x)$ of $G'$ has many infected neighbors in the configuration $\xi_{k\cdot R}\cap G'$, we have $\eta_k(x) = 1$. The key idea is that, on $G'$, around vertices of degree above $M$, the infection has high probability of persisting for more than $R$ units of time, and during this period, of propagating far enough that other vertices of high degree are reached; this is then interpreted as a transmission in the process $(\eta_k)$. Even if the parameter $\lambda$ is very small, we can construct $T$ and $\theta$ so that, if $n$ is large enough, $(\eta_k)$ has parameter $\lambda'$ larger than the critical parameter for the one-dimensional contact process. We then apply our results to conclude that $(\eta_k)$, and consequently $(\xi_t)$, survive for a long time.

\medskip

\noindent \textbf{Notations.} For $x \in \R$, we write $\lfloor x \rfloor$ for the integer part of $x$. If $A$ is a set, $|A|$ denotes its cardinality. When talking about the size of a graph, we always mean its number of vertices.

%
%
%
%
%%%%%%%%%%%%%%%%%%%%%%%%%%%%%%%%%%%%%%%%%%%%%%%%%%%%%%%%%%%%%%
%%%%%%%%%%%%%%%%%%%%%%%%%%%%%%%%%%%%%%%%%%%%%%%%%%%%%%%%%%%%%%
%
%
%
\section{A reminder on the contact process}
\label{s:remind}
\setcounter{equation}{0}

We start this section by presenting the graphical construction of the contact process and its self-duality property. Fix a graph $G = (V,E)$ and $\lambda > 0$. We take the following family of independent Poisson point processes on $[0,\infty)$:
$$\begin{array}{ll}
(D^x): x \in V &\text{with rate } 1;\\
(N^e): e \in E &\text{with rate } \lambda.\end{array}$$
Let $H$ denote a realization of all these processes. Given $x,y\in V,\; s \leq t$, we say that $x$ and $y$ are connected by an \textit{infection path in $H$} (and write $(x,s)\leftrightarrow (y,t)$ in $H$) if there exist times $t_0 = s < t_1 < \cdots < t_k = t$ and vertices $x_0 = x, x_1, \ldots, x_{k-1} = y$ such that
\begin{itemize}
\item[$\bullet$] $D^{x_i} \cap (t_i,\; t_{i+1}) = \emptyset$ for $i = 0, \ldots, k - 1$;
\item[$\bullet$] $\{x_i,x_{i+1}\}\in E$ for $i = 0, \ldots, k-2$;
\item[$\bullet$] $t_i \in N^{x_{i-1}, x_i}$ for $i = 1, \ldots, k-1$.
\end{itemize}
Points of the processes $(D^x)$ are called \textit{death marks} and points of $(N^e)$ are \textit{links}; infection paths are thus paths that traverse links and do not touch death marks. $H$ is called a \textit{Harris system}; we often omit dependence on $H$. For $A, B \subset V$, we write $A\times\{s\} \leftrightarrow B \times \{t\}$ if $(x,s)\leftrightarrow (y,t)$ for some $x \in A$, $y \in B$. We also write $A\times \{s\} \leftrightarrow (y,t)$ and $(x,s) \leftrightarrow B\times\{t\}$. Finally, given another set $C \subset V$, we write $A \times \{s\} \leftrightarrow B \times \{t\}$ \textit{inside} $C$ if there is an infection path from a point in $A\times \{s\}$ to a point in $B\times\{t\}$ and the vertices of this path are entirely contained in $C$.

Given $A \subset V$, put
\begin{equation}\label{eq1harris} \xi^A_t(x) = \mathds{1}_{\{A \times \{0\} \leftrightarrow (x,t)\}} \text{ for } x \in V,\; t \geq 0\end{equation}
(here and in the rest of the paper, $\mathds{1}$ denotes the indicator function). It is well-known that the process $(\xi^A_t)_{t\geq 0} = (\xi^A_t(H))_{t\geq 0}$ thus obtained has the same distribution as that defined by the infinitesimal generator (\ref{eq1gen}).  The advantage of (\ref{eq1harris}) is that it allows us to construct in the same probability space versions of the contact processes with all possible initial distributions. From this joint construction, we also obtain the \textit{attractiveness} property of the contact process: if $A \subset B \subset V$, then $\xi^A_t(H) \subset \xi^B_t(H)$ for all $t$. From now on, we always assume that the contact process is constructed from a Harris system, and will write $P_{G,\lambda}$ to refer to a probability measure under which such a system (on graph $G$ and with rate $\lambda$) is defined; we usually omit $G,\lambda$.

Now fix $A \subset V,\; t > 0$ and a Harris system $H$. Let us define the \textit{dual process} $(\hat \xi^{A, t}_s)_{0 \leq s \leq t}$ by
$$\hat \xi^{A,t}_s(y) = \mathds{1}_{\{(y,t-s)\leftrightarrow A \times \{t\} \text{ in } H\}}.$$
If $A = \{x\}$, we write $(\hat \xi^{x,t}_s)$.
This process satisfies two important properties. First, its distribution (from time 0 to $t$) is the same as that of a contact process with same initial configuration. Second, it satisfies the \textit{duality equation}
\begin{equation}\xi^A_t \cap B \neq \emptyset \text{ if and only if } A \cap \hat \xi^{B,t}_t \neq \emptyset. \end{equation}
In particular,
\begin{equation}\xi^{\underline 1}_t(x) = 1 \text{ if and only if } \hat \xi^{x,t}_t \neq \emptyset, \end{equation}
where $(\xi^{\underline 1}_t)$ is the process started from full occupancy.

\medskip

We now recall classical results about the contact process on an interval.
\begin{proposition}
\label{cpinterval}
For $n \in \Z_+$, $A \subset \Z_+$, let
$$
\sigma_n^A = \inf \Ll\{t \ge 0 : \xi_t^A(n) = 1\Rr\},
$$
where $(\xi_t^A)_{t \ge 0}$ denotes the contact process on $\Z_+$ with initial configuration $A$. For any $\lambda > \lambda_c(\Z)$, there exists $\ov{c}_1 > 0$, $n_0$ such that the following results hold.
\begin{enumerate}
\item
For any $n$,
$$
P\Ll[\sigma_n^{\{0\}} < \frac{n}{\ov{c}_1}\Rr] > \ov{c}_1.
$$
\item
For any $A \subset \{0,\ldots,n\}$ and any $n \ge n_0$,
$$
P\Ll[ \sigma_0^A + \sigma_n^A \ge \frac{n}{\ov{c}_1}, \ \xi^A_{n/\ov{c}_1} \neq \un{0} \Rr] \le e^{-n}.
$$
\item
If $(\xi_{n,t}^{\un{1}})_{t \ge 0}$ denotes the contact process on $\{0,\ldots,n\}$ started with full occupancy, then for any $n \ge n_0$ and any $t \ge 0$, we have
$$
P\Ll[ \xi_{n,t}^{\un{1}} = \un{0} \Rr] \le t e^{-\ov{c}_1 n}.
$$
\end{enumerate}
\end{proposition}
This follows from the classical renormalization argument that compares the contact process with supercritical oriented percolation, see for instance the proof of \cite[Corollary~VI.3.22]{lig85}.

%
%
%
%
%%%%%%%%%%%%%%%%%%%%%%%%%%%%%%%%%%%%%%%%%%%%%%%%%%%%%%%%%%%%%%
%%%%%%%%%%%%%%%%%%%%%%%%%%%%%%%%%%%%%%%%%%%%%%%%%%%%%%%%%%%%%%
%
%
%
\section{Metastability}
\label{s:metastab}
\setcounter{equation}{0}
We begin with the following basic graph-theoretic observation.
\begin{lemma}
\label{lem1}
For a tree $T \in \Lambda(n,d)$, there exists an edge whose removal separates
$T$ into two subtrees $T_1$ and $T_2$ both of size at least $\lfloor n/d \rfloor$.
\end{lemma}

\begin{proof}
Associate to each edge the value of the smallest cardinality of the two subtrees resulting from the edge's removal. Let $\{x,y\} $ be an edge having maximal value. We suppose that the subgraph $T_y$ containing vertex $y$ is the smaller and that the value of its subtree is less than $\lfloor n/d \rfloor - 1$.  Let the remaining edges of vertex $x$ be $\{x, x_1\}, \ \{x, x_2\}, \cdots \{x, x_r\} $, where $r \leq d-1$.  Let $T_j$ be the subtree containing $x_j$ obtained by removing the edge $\{x, x_j\}$, and let $n_j$ be its cardinality.  By maximality, all the $n_j$ must be less than $\lfloor n/d \rfloor - 1$, but equally,
$$
|T_y| = \Ll| T \setminus \Ll(\{x\} \cup T_1 \cup \cdots \cup T_r\Rr) \Rr| =
n-(1 + n_1 + n_2 + \cdots +n_r) \leq \lfloor n/d \rfloor - 1.
$$
That is, $n \le (d-1) (\lfloor n/d \rfloor - 1) + \lfloor n/d \rfloor \le n -(d-1)$, a contradiction (the case $d = 1$ being trivial).
\end{proof}

\begin{proposition}
\label{p:expalpha}
For any $\lambda > \lambda_c(\Z)$, there exists $\alpha > 0$ and $\ov{c}_2 > 0$ such that the following holds.
\begin{enumerate}
\item For any $n$ large enough, any $T \in \Lambda(n,d)$, any non-empty $A \subset T$, one has
$$
P\Ll[ \uptau^A \ge e^{\ov{c}_2 n^\alpha} \Rr] \ge \ov{c}_2.
$$
In particular, $E[\uptau^A] \ge \ov{c}_2 e^{\ov{c}_2 n^\alpha}$.
\item Moreover,
$$
P\Ll[ \uptau \ge e^{\ov{c}_2 n^{\alpha/2}} \Rr] \ge 1-e^{-\ov{c}_2 n^{-{\alpha/2}}},
$$
where we recall that we write $\uptau$ as a shorthand for $\uptau^{\un{1}}$.
\item
For $n$ large enough and any $G \in \cG(n,d)$, if the contact process on $G$ started with an arbitrary non-empty configuration survives up to time $n^2$, then the chance that at this time, it is equal to the contact process starting from full occupancy, is at least $1- e^{-n^{-\alpha/2}}$.
\end{enumerate}
\end{proposition}

From now on, $d$ is fixed and we consider a tree $T$ of maximal degree $d$ and size $n \to \infty$. Let $\beta > 0$ to be determined, not depending on $n$. Applying Lemma~\ref{lem1} repeatedly $\beta \log n$ times, we obtain $L_n = 2^{\beta \log n}$ disjoint subtrees each of size at least $\frac{n}{{(2d)}^{\beta \log n}} \geq \sqrt{n}$, provided $\beta \le 1/(2\log(2d))$ (for clarity, we simply assume that $L_n$ is an integer, without writing that the integer part should be taken). We write $T_1,\ldots, T_{L_n}$ for the trees thus obtained.

Since the tree $T$ has maximal degree bounded by $d$, so do the subtrees $(T_j)$. Now, the size of a tree with maximal degree $d$ is at most
$$
1+d+\ldots+d^\textsf{diam} = \frac{d^{\textsf{diam}+1}-1}{d-1},
$$
where $\textsf{diam}$ denotes its diameter. As a consequence, for $n$ large enough, each $T_j$ must have a diameter at least $\frac{\log n}{4 \log d}$, and thus contain a path of $\frac{\log n}{4 \log d}$ distinct vertices. We write $I_j$ to denote such a path, which we identify with an interval of length $\frac{\log n}{4 \log d}$.

\medskip

In what follows, we will distinguish between the two possibilities:
\begin{enumerate}
\item[(A)]
the diameter of $T$ is at least $n^{\alpha}$,
\item[(B)]
the diameter of $T$ is less than $n^\alpha$,
\end{enumerate}
where $\alpha > 0$ is a fixed number whose value will be specified in the course of the proof. It is worth keeping in mind that $\alpha$ will be chosen much smaller than $\beta$, itself chosen as small as necessary.

\begin{proof}[Proof of parts (1-2) of Proposition~\ref{p:expalpha}] Assume that the tree $T$ satisfies (A). For part (1), by attractiveness, it suffices to consider initial configurations with a single occupied site $z$.  Condition (A) ensures that one can find an interval of length at least $n^\alpha$. We write $[x,y]$ to denote such an interval, with $x$ and $y$ its endpoints.
Consider the event that within time $2n/\ov{c}_1$, the contact process has infected site $x$, and thereafter the contact process begun at this time restricted to $[x,y]$ and with only $x$ occupied has infected $y$. This event has probability at least $\ov{c}_1^2$ by part (1) of Proposition~\ref{cpinterval}. If this event occurs, then at time $2n/\ov{c}_1$, the contact process on $T$ dominates the contact process on $[x,y]$ begun with full occupancy. The desired bound now follows from bounds on survival times for supercritical contact processes on an interval, see part (3) of Proposition~\ref{cpinterval}. Part (2) also follows using the interval $[x,y]$ and part (3) of Proposition~\ref{cpinterval}.
% Let $B_1$ be the event that within time $n/{c_{0.2}} + \frac{\log n}{2 c_{0.2} \log d}$,
%the contact process started at $z$ has infected a site, $w$, within $[x,y]$ and thereafter
%the contact process begun at this time restricted to $[x,y]$ and with only $w$ occupied
%has infected both $x$ and $y$. This event has probability at least $c_{0.1}(c_{0.1}^2 - 2/n^{1/2 \log(d)})$.  If this event occurs, then at time $n/{c_{0.2}} + \frac{\log n}{2 c_{0.2} \log d}$, the contact process on $T$ dominates the contact process on $[x,y]$ begun with full occupancy.  The desired bound now follows from bounds on survival times for
%supercritical contact processes on intervals. (* add pointers to "basic facts" section *)
%

We now consider that the graph satisfies (B), and adapt an approach due to \cite{CD}. For any $A \subset I_i$, we write $(\xi^{A}_{i,t})_{t \ge 0}$ for the contact process on $I_i$ with initial configuration $A$, and define
$$
p_i(A) = P\Ll[ \xi^{A}_{i,Kn^\alpha} = \xi^{\un{1}}_{i,Kn^\alpha} \neq \un{0} \Rr],
$$
where $K = 2/\ov{c}_1$.
For any $i \le L_n$, we say that the interval $I_i$ is \emph{good at time} $t$ if $p_i(\xi_t) \ge 1 -n^{-2\beta}$,
where for simplicity we write $p_i(\xi_t)$ instead of $p_i\Ll({\xi_t} \cap I_i\Rr)$.
%
%$$P[\xi^{i}_{K n^\alpha} \not= 0]\geq 1 - \frac{1}{2 n^{10 \beta}} \quad \text{ and } \quad P[\xi^{i}_{K n^\alpha} = \xi_{K n^\alpha}^{ \underline 1, i}] \geq 1- \frac{1}{2 n^{10 \beta}},$$
%
%$$P(\xi^{i}_{s} \not= 0\hspace{0.3cm} \textrm{for}\hspace{0.3cm} s \leq \frac{2}{c_{0.2}} n^{\alpha}) \geq 1 - \frac{1}{2 n^{10 \beta}}$$
%and
%$$ \hspace{0.3cm} P(\xi^{i}_{s} = \xi_{s}^{ \underline 1, i} \hspace{0.3cm}\textrm{for} \hspace{0.3cm}s = \frac{2}{c_{0.2}} n^{\alpha}) \geq 1- \frac{1}{2 n^{10 \beta}}$$
%where $K = 2/c_{2}$, and $(\xi_{t}^{ \underline 1, i})_{t \ge 0}$ denotes the contact process on $I_i$ started with full occupancy.

For $k \in \N$, we let $X_k \in \{0,\ldots L_n\}$ be the number of good intervals at time $kKn^{\alpha}$. For $i \le L_n$ and $k \ge 0$, let us write $\cE_{i,k}$ for the event that the interval $I_i$ is good at time $kK n^\alpha$. By definition,
$$
P[\cE_{i,k+1} \ | \ \cE_{i,k}] = P[p_i(\xi_{(k+1)Kn^\alpha}) \ge 1 - n^{-2\beta} \ | \ \cE_{i,k}].
$$
By attractiveness, the latter is larger than
\begin{equation*}
\begin{split}
& P\Ll[p_i\Ll( \xi^{\xi_{kKn^\alpha}}_{i,Kn^\alpha} \Rr)\ge 1 - n^{-2\beta} \ | \ \cE_{i,k}\Rr] \\
& \qquad \ge P\Ll[p_i\Ll( \xi^{\un{1}}_{i,Kn^\alpha} \Rr) \ge 1 - n^{-2\beta}, \xi^{\xi_{kKn^\alpha}}_{i,Kn^\alpha} = \xi^{\un{1}}_{i,Kn^\alpha} \ | \ \cE_{i,k}\Rr] \\
& \qquad \ge 1 - P\Ll[p_i\Ll( \xi^{\un{1}}_{i,Kn^\alpha} \Rr) < 1 - n^{-2\beta}\Rr]
- \underbrace{P\Ll[\xi^{\xi_{kKn^\alpha}}_{i,Kn^\alpha} \neq \xi^{\un{1}}_{i,Kn^\alpha} \ | \ \cE_{i,k}\Rr]}_{\le n^{-2\beta}}.
\end{split}
\end{equation*}
We now argue that for $n$ large enough,
\begin{equation}
\label{e:expalpha}
P\Ll[p_i\Ll( \xi^{\un{1}}_{i,Kn^\alpha} \Rr) < 1 - n^{-2\beta}\Rr] \le n^{-2\beta}.
\end{equation}
Letting $(\xi^{A,t}_s)_{s \ge t}$ be the contact process started at time $t$ with $A$ occupied, one can rewrite the probability on the l.h.s.\ of \eqref{e:expalpha} as
\begin{multline*}
P\Ll[ P\Ll[\xi^{\xi^{\un{1}}_{i,Kn^\alpha},Kn^\alpha}_{i,2Kn^\alpha} \neq \xi^{\un{1},Kn^\alpha}_{i,2Kn^\alpha} \text{ or } \xi^{\un{1},Kn^\alpha}_{i,2Kn^\alpha} = \un{0} \ | \ \xi^{\un{1}}_{i,Kn^\alpha} \Rr] > n^{-2\beta} \Rr] \\
\le n^{2\beta} P\Ll[ \xi^{\un{1}}_{i,2Kn^\alpha} \neq \xi^{\un{1},Kn^\alpha}_{i,2Kn^\alpha} \text{ or }  \xi^{\un{1}}_{i,2Kn^\alpha} = \un{0}\Rr].
\end{multline*}
By part (3) of Proposition~\ref{cpinterval}, the contact process on $I_i$ started with full occupancy survives up to time $2Kn^\alpha$ with probability larger than
$$
1-2Kn^\alpha \exp\Ll(-\ov{c}_1 |I_i|\Rr) = 1-2Kn^{\alpha-\ov{c}_1/4\log d }.
$$
On this event, the probability that it gets coupled with the contact process started from full occupancy at time $Kn^\alpha$ within time $Kn^\alpha$ is larger than $1-e^{-|I_i|} = 1-n^{-\ov{c}_1/4\log d}$ by part (2) of Proposition~\ref{cpinterval}. Hence, the l.h.s.\ of \eqref{e:expalpha} is bounded by
$$
n^{2\beta} \Ll( 2Kn^{\alpha-\ov{c}_1/4\log d } + n^{-\ov{c}_1/4\log d} \Rr),
$$
which can be made smaller than $n^{-2\beta}$ if $0 < \alpha \ll \beta \ll 1$ are suitably chosen. To sum up, we have shown that for all $n$ large enough,
$$
P[\cE_{i,k+1} \ | \ \cE_{i,k}] \ge 1-2 n^{-2\beta}.
$$
Moreover, an examination of the above proof shows that this estimate still holds if we condition also on the state of the intervals $(I_j)_{j \neq i}$. In other words, we have shown that for any $x \ge 0$,
\begin{equation}
\label{driftgauche}
P\Ll[X_{k+1} \leq X_{k} -x \ | \ X_{k}\Rr] \leq P\Ll[\mathsf{Bin}(L_n, 2 n^{-2\beta}) \geq x\Rr],
\end{equation}
where $\mathsf{Bin}(n,p)$ denotes a binomial random variable of parameters $n$ and $p$. Note also that with probability tending to $1$, all the intervals that are good at time $kKn^\alpha$ remain so at time $(k+1)Kn^\alpha$.

We now show that if $l < L_n$, then
\begin{equation}
\label{driftdroite}
P\Ll[X_{k+1}-X_k \ge 1 \ | \ \xi_{kKn^\alpha} \neq \un{0}, X_k = l \Rr] \ge \frac{\ov{c}_1^2}{2}.
\end{equation}
(Obviously, if $X_k = l \neq 0$, then it must be that $\xi_{kKn^\alpha} \neq \un{0}$.) By the Markov property, it suffices to show \eqref{driftdroite} for $k = 0$. We thus consider a non-empty initial configuration $A$ with $l < L_n$ good intervals. Let $I_i = [x,y]$ be an interval that is not good at time $0$. With probability tending to $1$, all good intervals remain good at time $Kn^\alpha$, so we only need to study the probability that $I_i$ becomes good. The probability of the complementary event is
$$
P\Ll[ p_i\Ll(\xi^A_{Kn^\alpha}\Rr) < 1-n^{-2\beta} \Rr] \le P\Ll[\xi^A_{Kn^\alpha} < \xi_{i,Kn^\alpha}^{\un{1}} \Rr] + P\Ll[ p_i\Ll(\xi_{i,Kn^\alpha}^{\un{1}}\Rr) < 1-n^{-2\beta} \Rr] .
$$
Inequality \eqref{e:expalpha} ensures that the last probability becomes arbitrarily small for $n$ large enough. It thus suffices to show that
\begin{equation}
\label{e:expal}
P\Ll[\xi^A_{Kn^\alpha} < \xi_{i,Kn^\alpha}^{\un{1}} \Rr] \le 1-\ov{c}_1^2.
\end{equation}
Let $z \in A$. We consider the event $\cE_1$ that within time $Kn^\alpha = 2n^\alpha/\ov{c}_1$, the contact process has infected $x$, and thereafter the contact process restricted to $[x,y]$ and with only $x$ occupied has reached $y$. Note that the diameter of $T$ is less than $n^\alpha$ (so that there exists a path of length less than $n^\alpha$ linking $z$ to $x$), while the length of $I_i$ is $\frac{\log n}{4 \log d} \le n^\alpha$. As a consequence, part~(1) of Proposition~\ref{cpinterval} ensures that the event $\cE_1$ has probability at least $\ov{c}_1^2$. Since on the event $\cE_1$, we have $\xi^A_{Kn^\alpha} \ge \xi_{i,Kn^\alpha}^{\un{1}}$, this justifies \eqref{e:expal}, and thus also \eqref{driftdroite}.

The conclusion will now follow from \eqref{driftgauche} and \eqref{driftdroite} by a comparison with a random walk on $\Z \cap (-\infty,L_n]$ with a drift to the right. The necessary information on this drifted walk is contained in the following lemma.
\begin{lemma}
\label{l:rw}
Let $(Z_l)_{l \in \N}$ be the random walk on $\Z \cap (-\infty,L_n]$ with transition probabilities
$$
P[Z_{l+1} = x + k \ | \ Z_l = x < L_n] =
\left|\begin{array}{ll}
0 & \text{if } k > 1, \\
\ov{c}_1^2/{2} & \text{if } k = 1, \\
e^{-n^{-\beta}} \ n^{-|k|\beta}/{|k|!} & \text{if } k \le -1.
\end{array}
\right.
$$
Let also $H_0$ be the hitting time of $\Z_- = \Z \cap(-\infty,0]$, and $H_L$ be the hitting time of $L_n$. For any $n$ large enough and any $x \le L_n$, we have
$$
P\Ll[ H_0 < H_L \ | \ Z_0 = x \Rr] \le n^{-x \beta/2}.
$$
\end{lemma}

Let us postpone the proof of this lemma, and see how it enables us to conclude. From \eqref{driftdroite}, we learn that whatever the initial non-empty configuration, we have $X_1 \ge 1$ with probability bounded away from $0$. On this event, we want to couple $(X_k)$ with the random walk of the lemma, so that $X_{k-1} \ge Z_k$ for every $k \ge 0$. In the r.h.s.\ of \eqref{driftgauche}, a binomial random variable appears, while jumps to the left in the lemma follow a Poisson random variable. Since a Bernoulli random variable of parameter $p$ is stochastically dominated by a Poisson random variable of parameter $-\log(1-p)$, it follows that $\mathsf{Bin}(L_n, 2 n^{-2\beta})$ is stochastically dominated by a Poisson random variable of parameter
$$
-L_n \log (1-2 n^{-2\beta})  = -n^{\beta \log 2}\log (1-2 n^{-2\beta}) \le n^{-\beta}.
$$
This and \eqref{driftdroite} guarantee the existence of the coupling. With probability at least $1-n^{-\beta/2} \ge 1/2$, the random walk hits $L_n$ before entering $\Z_-$. The proof of part (1) will be complete if we can argue that starting from $L_n$, with probability close to $1$, the walk needs to exit $L_n$ at least $e^{n^{\alpha}}$ times before reaching $\Z_-$. Let us consider a sequence of $e^{n^\alpha}$ excursions from $L_n$, and show that with high probability, none of them visits $\Z_-$. The first jump out of $L_n$ is distributed according to a Poisson random variable of parameter $n^{-\beta}$, which (for convenience) may be dominated by an exponential random variable of parameter $1$. With probability tending to $1$, the maximum over $e^{n^{\alpha}}$ such random variables does not exceed $n^{2 \alpha} \le L_n/4$. In view of the lemma, given an excursion whose first step has size smaller than $L_n/4$, the excursion will visit $\Z_-$ with probability smaller than $n^{-3L_n\beta/4} \le e^{-2n^\alpha}$, and this finishes the proof of part (1).

As for part (2), the argument is similar, except that in this case $X_0 = L_n$. Consider $e^{n^{\alpha/2}}$ excursions from $L_n$. With probability at least $1-e^{-n^{\alpha/2}}$, none of these excursions has size larger than $n^{2\alpha} \le L_n/4$. As noted above, given an excursion from $L_n$ whose first step has size smaller than $L_n/4$, the excursion will visit $\Z_-$ with probability smaller than $n^{-3L_n\beta/4} \le e^{-2n^\alpha}$, thus finishing the proof of part (2).
\end{proof}

\begin{proof}[Proof of Lemma~\ref{l:rw}]
Let $h(x) = P\Ll[ H_0 < H_L \ | \ Z_0 = x \Rr]$, $\td{h}(x) = n^{-x \beta/2}$, and let $\cL$ be the generator of the random walk:
$$
\cL f(x) = \frac{\ov{c}_1^2}{2}(f(x+1)-f(x)) + e^{-n^{-\beta}} \sum_{k = 1}^{+\infty}  \frac{n^{-k\beta}}{k!} (f(x-k) - f(x))\qquad (x < L_n).
$$
For $x \in \Z \cap (0,L_n)$, we have $\cL h(x) = 0$. On the other hand, for such $x$, we have
\begin{eqnarray*}
\cL \td{h}(x) & = & \frac{\ov{c}_1^2}{2}\Ll(n^{-\beta/2}-1\Rr)\td{h}(x) + e^{-n^{-\beta}}\sum_{k = 1}^{+\infty}  \frac{n^{-k\beta}}{k!} (n^{k\beta/2}-1)\td{h}(x) \\
& \le & \frac{\ov{c}_1^2}{2}\Ll(n^{-\beta/2}-1\Rr)\td{h}(x) + \sum_{k = 1}^{+\infty}  \frac{n^{-k\beta}}{k!} n^{k\beta/2}\td{h}(x) \\
& \le & \Ll[\frac{\ov{c}_1^2}{2}\Ll(n^{-\beta/2}-1\Rr) + e^{n^{-\beta/2}}-1  \Rr] \td{h}(x),
\end{eqnarray*}
so $\cL \td{h}(x)\le 0$ as soon as $n$ is large enough. As a consequence, $\cL (h-\td{h}) \ge 0$ on $\Z \cap (0,L_n)$. By the maximum principle,
$$
\max_{\Z \cap (0,L_n)} (h-\td{h}) \le \max_{\Z_- \cup \{L_n\}} (h-\td{h}) = 0,
$$
and the lemma is proved.
\end{proof}
The following observation will be useful in the proof of part (3) of Proposition~\ref{p:expalpha}.
\begin{rem}
\label{remduality}	
Let a Harris system for the contact process on some graph $G = (V,E)$ be given (and fixed). We identify $G$ with its set of vertices, and assume that $\xi^A_t = \xi^{\un{1}}_t$ for some $A \subset V$ and $t > 0$. This implies that in the Harris system, any infection path from $V\times \{0\}$ to $V\times \{t\}$ intersects the offspring of elements of $A$. Let $(\hat{\xi}^{B,t}_s)_{0 \le s \le t}$ be the dual contact process for time $t$, started with configuration $B$. If furthermore, $\hat{\xi}^{B,t}$ survives up to time $t$, then there must exist an infection path from $A \times \{0\}$ to $B \times \{t\} $.
\end{rem}

\begin{proof}[Proof of part (3) of Proposition~\ref{p:expalpha}]
We continue with case (B), but considering that $T$ is the spanning tree of some graph $G = (V,E)$. For an arbitrary $z \in V$, we wish to bound
$$
P\Ll[ \xi^z_{n^2} \neq \xi^{\un{1}}_{n^2}, \ \xi^z_{n^2} \neq \un{0} \Rr].
$$
The probability above is equal to $P[\exists y : \xi^z_{n^2}(y) \neq \xi^{\un{1}}_{n^2}(y), \ \xi^z_{n^2} \neq \un{0}]$. For any fixed $y$, we will thus bound
\begin{equation}
\label{e:debut}
P\Ll[\xi^z_{n^2}(y) \neq \xi^{\un{1}}_{n^2}(y), \ \xi^z_{n^2} \neq \un{0}\Rr].
\end{equation}
Letting $(\hat{\xi}^{y,n^2}_t)_{0 \le t \le n^2}$ be the dual contact process for time $n^2$ started with configuration $\{y\}$, we can rewrite this probability as
$$
P\Ll[\xi^z_{n^2}(y) = 0,\  \hat{\xi}^{y,n^2}_{n^2} \neq \un{0}, \ \xi^z_{n^2} \neq \un{0}\Rr].
$$
As in the proof of part (1), we consider $X_k$ the number of good intervals at time $kKn^\alpha$. By attractiveness, if an interval is good for the contact process on $T$, then it must be good for the contact process on $G$. Note that, for $H_L$ as in Lemma~\ref{l:rw}, a classical large deviation estimate on sums of i.i.d.\ random variables with an exponential moment gives us that
$$
P\Ll[ H_L > n \Rr] \le e^{-\sqrt{n}},
$$
and as a consequence,
\begin{equation}
\label{step1}
P\Ll[L_n \notin \{X_k, \ k \le n\}, \ \xi^z_{nKn^\alpha} \neq \un{0} \Rr] \le e^{-\sqrt{n}}.
\end{equation}
Let $\cE_{3/4}$ be the event that starting from $z$ occupied, at least $3/4$ of all the intervals~$(I_i)_{i \le L_n}$ are good at time $n^2/2$ (which, for simplicity, is assumed to be a multiple of $Kn^\alpha$). As the proof of part (2) reveals, once $X_k$ has reached $L_n$, the probability that it makes an excursion below $3L_n/4$ before time $n^2$ is smaller than $e^{-n^{\alpha}}$. Combining this with \eqref{step1}, we obtain
$$
P\Ll[ \xi^z_{n^2} \neq \un{0},\ \cE_{3/4}^c \Rr] \le 2 e^{-n^{\alpha}},
$$
where $\cE_{3/4}^c$ denotes the complement of $\cE_{3/4}$. Similarly, if we let $\hat{\cE}_{3/4}$ denote the event that for the dual process $\hat{\xi}^{y,n^2}$, at least $3/4$ of the intervals are good at time $n^2/2 - Kn^\alpha$, then
$$
P\Ll[ \hat{\xi}^{y,n^2}_{n^2} \neq \un{0},\ \hat{\cE}_{3/4}^c \Rr] \le 2 e^{-n^{\alpha}}.
$$
Consider the event $\td{\cE}_i$ defined by:
\begin{equation*}
\begin{array}{l}
\text{during the time interval } [n^2/2,n^2/2+Kn^\alpha]\text{, the direct contact process}\\
\text{restricted to } I_i \text{ becomes identical with the contact process started with full} \\
\text{occupancy (on } I_i \text{), while the dual contact process restricted to } I_i \text{ survives}.
\end{array}
\end{equation*}
%With the notation of the proof of parts (1-2) of Proposition~\ref{p:expalpha}, the event $\td{\cE}_i$ can be rephrased as
%$$
%\xi_{i,Kn^\alpha}^{\xi_{n^2/2}^z} = \xi_{i,Kn^\alpha}^{\un{1}} \text{ and } \hat{\xi}_{i,Kn^\alpha}^{\hat{\xi}^{y,n^2}_{n^2/2-Kn^\alpha}} \neq \un{0}.
%$$
Let also $\cI$ be the set of indices $i$ such that $I_i$ is good both for the contact process and its dual. We have
\begin{equation*}
P\Ll[\bigcap_{i \le L_n} (\td{\cE}_i)^c, \cE_{3/4},\hat{\cE}_{3/4} \Rr] \le
P\Ll[\bigcap_{i \in \cI} (\td{\cE}_i)^c, \cE_{3/4},\hat{\cE}_{3/4} \Rr].
\end{equation*}
Given that ${\cE}_{3/4}$ and $\hat{\cE}_{3/4}$ both happen, at least $1/2$ of the intervals are good both for the contact process and its dual, or in other words, $|\cI| \ge L_n/2$. Moreover, the events $\cE_{3/4}$ and $\hat{\cE}_{3/4}$, and the set $\cI$, are independent of the state of the Harris system in the time layer $T \times [n^2/2,n^2/2+Kn^\alpha]$. By the definition of being good, we have $P[(\td{\cE}_i)^c \ | \ i \in \cI] \le 2 n^{-2\beta}$.
Note also that the events $(\td{\cE}_i)$ are independent. Hence
\begin{equation*}
P\Ll[\bigcap_{i \le L_n} (\td{\cE}_i)^c, \cE_{3/4},\hat{\cE}_{3/4} \Rr] \le
(2n^{-2\beta})^{L_n/2}.
\end{equation*}
Finally, note that when one of the $\td{\cE}_i$ happens, it must be that $\xi^z_{n^2}(y) = 1$, by Remark~\ref{remduality}. We have thus proved that
\begin{eqnarray*}
P\Ll[\xi^z_{n^2}(y) = 0,\  \hat{\xi}^{y,n^2}_{n^2} \neq \un{0}, \ \xi^z_{n^2} \neq \un{0}\Rr] & \le & P\Ll[\xi^z_{n^2}(y) = 0,\ \cE_{3/4}, \hat{\cE}_{3/4}  \Rr] + 4 e^{-n^{\alpha}}\\
& \le & \Ll( 2n^{-2\beta} \Rr)^{L_n/2} + 4 e^{-n^{\alpha}} \\
& \le & 5 e^{-n^{\alpha}}.
\end{eqnarray*}
Recalling that the probability on the l.h.s.\ above is that appearing in \eqref{e:debut}, we have thus shown that
$$
P\Ll[ \xi^z_{n^2} \neq \xi^{\un{1}}_{n^2}, \ \xi^z_{n^2} \neq \un{0} \Rr] \le 5n e^{-n^{\alpha}}.
$$
Now for a general $A \subset V$, we have
$$
P\Ll[ \xi^A_{n^2} \neq \xi^{\un{1}}_{n^2}, \ \xi^A_{n^2} \neq \un{0} \Rr] \le \sum_{z \in T } P\Ll[ \xi^z_{n^2} \neq \xi^{\un{1}}_{n^2}, \ \xi^z_{n^2} \neq \un{0} \Rr] \le 5n^2 e^{-n^{\alpha}}.
$$
In view of part (1) of Proposition~\ref{p:expalpha}, we thus have, for $A \neq \emptyset$,
$$
P\Ll[ \xi^A_{n^2} \neq \xi^{\un{1}}_{n^2} \ | \  \xi^A_{n^2} \neq \un{0} \Rr] \le \frac{5n^2}{\ov{c}_2} e^{-n^{\alpha}},
$$
which proves the desired result.

\medskip

For case (A), the reasoning is similar, only simpler. Let $I$ be an interval of length $n^\alpha$ contained in $T$. For any $A \subset I$, we write $(\xi^{A}_{I,t})_{t \ge 0}$ for the contact process on $I$ with initial configuration $A$, and define
$$
p(A) = P\Ll[ \xi^{A}_{I,K n} = \xi^{\un{1}}_{I,K n} \neq \un{0} \Rr].
$$
We say that $I$ is \emph{good at time} $t$ if $p(\xi_t) \ge 1-e^{-n^{3\alpha/4}}$, and for $k \in \N$, we let $X_k$ be the indicator function that $I$ is good at time $kK n$.

In view of the proof of part (1) of Proposition~\ref{p:expalpha}, we have
\begin{equation}
\label{driftright}
P\Ll[ X_{k+1} = 1 \ | \ \xi_{kKn} \neq \un{0} \Rr] \ge \ov{c}_1^2,
\end{equation}
while the same reasoning as in case (B) leads to
\begin{equation}
\label{driftleft}
P\Ll[X_{k+1} = 1 \ | \ X_k = 1 \Rr] \ge 1-2e^{-n^{3\alpha/4}}.
\end{equation}
From \eqref{driftright} and \eqref{driftleft}, one can see that, for any $z \in V$,
$$
P\Ll[ \xi_{n^2}^z \neq \un{0},\ I \text{ not good at time } n^{3/2} \text{ for } \xi^z \Rr] \le 2 e^{-n^{5\alpha/8}},
$$
where for simplicity we assume that $n^{3/2}$ is a multiple of $Kn$. Similarly, for any $z \in V$, one has
$$
P\Ll[ \hat{\xi}^{y,n^2}_{n^2} \neq \un{0},\  I \text{ not good at time } n^{3/2}-Kn \text{ for } \hat{\xi}^{y,n^2} \Rr] \le 2 e^{-n^{5\alpha/8}},
$$
and we conclude as in case (B).
\end{proof}

\begin{proof}[Proof of Theorem~\ref{thm2main2}]
The result follows from \cite[Proposition~2.1]{tommeta}, using parts (2-3) of Proposition~\ref{p:expalpha}.
\end{proof}

%$$P(\# \textrm{of infected sites by time}\hspace{0.3cm} n^{\alpha} \geq c n^{\beta \log 2}) > c$$
%
%for some $c > 0$ and thus (arguing as in (1))

%$$P(\# \textrm{of infected sites by time}\hspace{0.3cm} n^{\beta} \textrm{will be} \geq %\frac{9 n^{\beta}}{10}) \textrm{ outside}.$$

%exponentially small probability in $n$)

%\noindent Thus $\forall \hspace{0.3cm}z \in G \hspace{0.4cm} P(\hat{\xi} ^{z}_{n-n ^{2 \beta %}
 %} \cap \xi _{n^ [2\beta}] = \emptyset ,\hat{\xi} ^{z}_{n-n ^{2 \beta }
 %} , \xi _{n^ {2\beta} } \not= \emptyset ) \leq e^{- \frac{n^{\beta}}{4}}$

%from which the result follows.

\vspace{0.4cm}

%
%
%
%
%%%%%%%%%%%%%%%%%%%%%%%%%%%%%%%%%%%%%%%%%%%%%%%%%%%%%%%%%%%%%%
%%%%%%%%%%%%%%%%%%%%%%%%%%%%%%%%%%%%%%%%%%%%%%%%%%%%%%%%%%%%%%
%
%
%
\section{Comparison with Phoenix contact processes}
\label{s:comparison2}
\setcounter{equation}{0}

The aim of this section is to prove Theorems~\ref{thm1main1} and \ref{thm3main3}. To this end, we manufacture a ``Phoenix contact process''. This process evolves as a contact process up to extinction, but has then the ability to recover activity, making it a positive recurrent Markov process. Separating a tree $T$ into $T_1$ and $T_2$ as in Lemma~\ref{lem1}, we then show that with high probability, the true contact process $\xi$ dominates the union of the two Phoenix contact processes running independently on $T_1$ and $T_2$, and this enables us to conclude.

\medskip

Let $T \in \Lambda(n,d)$. Given a Harris system for the contact process on $T$, for any $x \in T$ and $t \ge 0$, we write $(\xi^{x,t}_s)_{s \ge t}$ for the contact process starting at time $t$ with $x$ the only occupied site.
We say that the Harris system is \emph{trustworthy} on the time interval $[0,n^4]$ if for any $(x,s) \in T \times [0,n^4/2]$, the following two conditions hold:
\begin{enumerate}
\item[(C$_1$)]
if $\xi^{x,s}$ survives up to time $n^4$, then $\xi^{x,s}_{n^4} = \xi^{\un{1}}_{n^4}$,
\item[(C$_2$)]
if $\xi^{x,s}$ survives up to time $s+2 n^2$, then it survives up to time $n^4$.
\end{enumerate}

We say that the Harris system $H$ is trustworthy on the time interval $[t,t+n^4]$ if $\Theta_t H$ is trustworthy on the time interval $[0,n^4]$, where $\Theta_t H$ is the Harris system obtained by a time translation of $t$.

For a given Harris system and for $(Y_t)_{t \in \R_+}$ a family of independent auxiliary random variables following a Bernoulli distribution of parameter $1/2$, independent of the Harris system, we define the \emph{Phoenix contact process} $(\zeta_{T,t})_{t \ge 0} = (\zeta_t)_{t \ge 0}$ on $\{0,1\}^T$ as follows.

\noindent \emph{Step 0.} Set $\zeta_0 = \un{1}$, and go to Step 1.

\noindent \emph{Step 1.} Evolve as a contact process according to the Harris system, up to reaching the state $\un{0}$, and go to Step 2.

\noindent \emph{Step 2.} Let $t$ be the time when Step 2 is reached. Stay at $\un{0}$ up to time $t+n^4$ and
\begin{itemize}
\item
if the Harris system is trustworthy on $[t,t+n^4]$ and $Y_{t} = 1$, then set $\zeta_{t+n^4} = \xi^{\un{1},t}_{t+n^4}$ (where $\xi^{\un{1},t}$ is the contact process started with full occupancy at time $t$ and governed by the Harris system), and go to Step 1 ;
\item
else, go to Step 2.
\end{itemize}

We say that the process is \emph{active} when it is running Step 1 ; is \emph{quiescent} when it is running Step 2. Note that after initialization, the process alternates between active and quiescent phases. If it happens that during Step 2, the Harris system is trustworthy on $[t,t+n^4]$ and $Y_t = 1$, but $\xi^{\un{1},t}_{t+n^4} = \un{0}$, we consider that the process is active at time $t+n^4$, and becomes inactive again immediately afterwards.

%\begin{rem}
%Since the time the process spends on state $\un{0}$ is not exponential, $(\zeta_t)$ is not Markovian. It is however easy to make it into a Markovian process, by enlarging its state space into $\Ll(\{0,1\}^T \setminus \{\un{0}\}\Rr) \cup \Ll(\{\un{0}\} \times [0,n^4)\Rr)$, so that when arriving in Step 2, the process is in the state $(\un{0},0)$, and subsequently the second coordinate increases at unit speed. We will always identify all points in $\{\un{0}\} \times [0,n^4)$ to the single state $\un{0}$.
%\end{rem}
\begin{rem}
\label{r:markov}
Note that since the time the process spends on state $\un{0}$ is not exponential, $(\zeta_t)$ is not Markovian. It would however be easy to make the process Markovian, by enlarging its state space into $\Ll(\{0,1\}^T \setminus \{\un{0}\}\Rr) \cup \Ll(\{\un{0}\} \times [0,n^4)\Rr)$, so that when arriving in Step 2, the process is in the state $(\un{0},0)$, and subsequently the second coordinate increases at unit speed.
\end{rem}

\begin{rem}
\label{r:randomization}
The auxiliary randomization of $\zeta$ provided by the family $(Y_t)$ is a technical convenience, which guarantees that if $\zeta_t$ is quiescent at some time $t$, then with probability at least $1/2$ it remains so at least up to time $t + n^4$.	
\end{rem}

\begin{rem}
\label{r:defnu}
Each time the process becomes active again, its distribution at this time is that of $\xi^{\un{1}}_{n^4}$ conditionned on the event that the Harris system is trustworthy on the time interval $[0,n^4]$. We write $\nu$ to denote this distribution.
\end{rem}

\begin{lemma}
\label{l:trust}
Let $T \in \Lambda(n,d)$. For any $n$ large enough and any $t$, the probability that the Harris system on $T$ is trustworthy on $[t,t+n^4]$ is larger than $1/2$.
\end{lemma}
\begin{proof}
It suffices to show the lemma for $t = 0$. We first consider condition (C$_1$).
By part (3) of Proposition~\ref{p:expalpha}, the probability that
\begin{equation}
\label{eventC1}
\forall z \in T, \ \xi^{z,n^4/2}_{n^4} \neq \un{0} \Rightarrow \xi^{z,n^4/2}_{n^4} = \xi^{\un{1},n^4/2}_{n^4}
\end{equation}
goes to $1$ as $n$ tends to infinity. Let $(x,s) \in T \times [0,n^4/2]$, and assume that $\xi^{x,s}$ survives up to time $n^4$, that is,
$$
(x,s) \leftrightarrow T \times \{n^4\}.
$$
Then there must exist $z \in T$ such that
$$
(x,s) \leftrightarrow (z,n^4/2) \leftrightarrow T \times \{n^4\}.
$$
On the event \eqref{eventC1}, we thus have $\xi^{x,s}_{n^4} \ge \xi^{\un{1},n^4/2}_{n^4}$. The converse comparison being clearly satisfied, we have in fact $\xi^{x,t}_{n^4} = \xi^{\un{1},n^4/2}_{n^4}$. In order to show that condition (C$_1$) is satisfied for any $(x,s) \in T \times [0,n^4/2]$ with probability tending to $1$, it thus suffices to show that
\begin{equation}
\label{C1bis}
P\Ll[ \xi^{\un{1}}_{n^4} = \xi^{\un{1},n^4/2}_{n^4} \Rr] \to 1 \text{ as } n \to \infty.
\end{equation}

In view of part (2) of Proposition~\ref{p:expalpha}, with probability tending to one, we have $\xi^{\un{1}}_{n^4} \neq \un{0}$. On this event, by part (3) of Proposition~\ref{p:expalpha}, we also have $\xi^{\un{1},n^4/2}_{n^4} = \xi^{\un{1}}_{n^4}$ with probability tending to $1$, and thus \eqref{C1bis} is proved.

We now turn to condition (C$_2$). Note that the event $\xi_{x,s}^{s+2n^2} \neq \un{0}$ can be rewritten as
$$
(x,s) \leftrightarrow T \times \{s+2n^2\},
$$
and under such a circumstance, there must exist $z \in T$ such that
$$
(x,s) \leftrightarrow (z, \lceil s/n^2 \rceil n^2) \leftrightarrow T \times \{s+2n^2\}.
$$
It is thus sufficient to show that
\begin{equation}
\label{C2a}
P\Ll[\exists z \in T, k \in \{0,\ldots, \lceil n^2/4 \rceil\} : \ \xi^{z,kn^2}_{(k+1)n^2} \neq \un{0} \text{ but } \xi^{z,kn^2}_{n^4} = \un{0}\Rr] \to 0 \text{ as } n \to \infty.
\end{equation}
For a fixed $z \in T$ and integer $k$, we have by part (3) of Proposition~\ref{p:expalpha} that
$$
P\Ll[ \xi^{z,kn^2}_{(k+1)n^2} \neq \un{0} \text{ but } \xi^{z,kn^2}_{(k+1)n^2}  \neq \xi^{\un{1},kn^2}_{(k+1)n^2} \Rr] \le e^{-n^{\alpha/2}},
$$
so the probability of the event
\begin{equation}
\label{C2b}
\forall z \in T, k \in \{0,\ldots, \lceil n^2/4 \rceil\} : \ \xi^{z,kn^2}_{(k+1)n^2} = \un{0} \text{ or } \xi^{z,kn^2}_{(k+1)n^2} = \xi^{\un{1},kn^2}_{(k+1)n^2}
\end{equation}
tends to $1$ as $n$ tends to infinity. On the other hand, with probability tending to $1$, $\xi^{\un{1}}$ survives up to time $n^4$, and is clearly dominated by $\xi^{\un{1},kn^2}_{(k+1)n^2}$, for any $k \le \lceil n^2/4 \rceil$. On the conjunction of this event and the one described in \eqref{C2b}, we thus have
$$
\forall z \in T, k \in \{0,\ldots, \lceil n^2/4 \rceil\} : \ \xi^{z,kn^2}_{(k+1)n^2} = \un{0} \text{ or } \xi^{z,kn^2}_{(k+1)n^2} \ge \xi^{\un{1}}_{n^4} \neq \un{0},
$$
and this proves \eqref{C2a}.
\end{proof}

\begin{lemma}
\label{l:attract}
For any $s > 0$, one has
$$
P\Ll[\uptau \le s \Rr] \le \frac{s}{s+E[\uptau]},
$$	
where we recall that $\uptau$ is the extinction time of the contact process started with full occupancy. Moreover, there exists a constant $C$ such that for any $T \in \Lambda(n,d)$, $E[\uptau] \le e^{Cn}$.
\end{lemma}
\begin{proof}
Attractiveness of the contact process implies that for any $r \in \N$,
\begin{equation}
\label{e:prob00}
P\Ll[ \uptau \ge r s \Rr] \le \Ll(P\Ll[ \uptau \ge s \Rr]\Rr)^r.
\end{equation}
Since
\begin{equation}
\label{e:prob1}
E[\uptau] \le s \sum_{r = 1}^{+\infty} P\Ll[ \uptau \ge r s \Rr] \le s \frac{P\Ll[ \uptau \ge s \Rr]}{1-P\Ll[ \uptau \ge s \Rr]},
\end{equation}
it comes that
$$
P\Ll[ \uptau \ge s \Rr] \ge \frac{E[\uptau]}{s+E[\uptau]},
$$
which proves the first part. For the second part, note that one can find $C$ such that
\begin{equation}
\label{e:prob01}
P\Ll[\uptau \ge 1\Rr] \le 1-e^{-Cn}
\end{equation}
uniformly over $T \in \Lambda(n,d)$. The conclusion thus follows from \eqref{e:prob1}.
\end{proof}

\begin{lemma}
\label{l:prob0}
For any $n$ large enough, any $T \in \Lambda(n,d)$ and any $t \ge 0$,
one has
\begin{equation}
\label{e:prob0}
P\Ll[ \zeta_t = \un{0} \Rr] \le \frac{6 n^6}{E[\uptau]}.
\end{equation}
\end{lemma}
\begin{proof}
Using Lemma~\ref{l:attract} with $s = n^6$, it is clear that \eqref{e:prob0} holds for any $n$ and any $t \le n^6$.
Note moreover that, writing $\uptau^\nu$ for the extinction time of the contact process started from the distribution $\nu$ defined in Remark~\ref{r:defnu}, we have
\begin{equation}
\label{e:prob2}
P\Ll[ \uptau^\nu \ge n^6 -n^4 \Rr] = P\Ll[ \uptau \ge n^6 \ | \ \text{Harris sys. trustworthy on } [0,n^4] \Rr] \le \frac{2n^6}{E[\uptau]},
\end{equation}
where we used Lemma~\ref{l:trust} in the last step.

Suppose now that $t > n^6$, and consider the event $\cE$ defined by
$$
\exists s \in (t - n^6/2, t - n^6/4] \mbox{ such that } \zeta_s = \underline 0.
$$
We write $\td{\uptau}$ for the first $s \ge t-n^6/2$ such that $\zeta_s = \un{0}$. On the event $\cE$, we have $\td{\uptau} \le t - n^6/4$. The event $\cE'$ defined by
\begin{multline*}
\forall k \in \N, k < \lfloor n^2/4 \rfloor , \\
 \text{Harris sys. not trustworthy on } [\td{\uptau} + k n^4, \td{\uptau} + (k+1) n^4] \text{ or } Y_{\td{\uptau} + k n^4} \neq 1
\end{multline*}
has probability smaller than $(3/4)^{\lfloor n^{2}/4 \rfloor}$ by Lemma~\ref{l:trust}. When $\cE$ and $(\cE')^c$ both hold, the process $\zeta$ becomes active at some time $t_A \in [t-n^6/2, t]$, and is distributed according to $\nu$ at this time. Hence,
\begin{eqnarray*}
P\Ll[\zeta_t = \un{0}, \cE \Rr] & \le & P\Ll[\zeta_t = \un{0}, \cE, (\cE')^c\Rr] + P\Ll[\cE'\Rr] \\
& \le & P\Ll[\zeta_t = \un{0}, \cE, (\cE')^c\Rr] + P\Ll[\cE'\Rr] \\
& \le & P\Ll[\uptau^\nu \le n^6/2\Rr] + P\Ll[\cE'\Rr].
\end{eqnarray*}
Since $P[\cE'] \ll 1/E[\uptau]$ and in view of \eqref{e:prob2}, we have indeed
\begin{equation}
\label{e:prob3}
P\Ll[\zeta_t = \un{0}, \cE \Rr] \le \frac{3 n^6}{E[\uptau]}
\end{equation}
for any large enough $n$. It thus remains to bound
\begin{equation}
\label{e:prob4}
P\Ll[\zeta_t = \un{0}, \cE^c \Rr].
\end{equation}
Let $k$ be the first positive integer such that $Y_{t-n^6/2+kn^4} = 1$ and the Harris system is trustworthy on
$$
[a_k,b_k] \stackrel{\text{(def)}}{=} [t - n^6/2 + k n^4, t - n^6/2 + (k+1) n^4].
$$
For the same reason as above, we may assume that $[a_k,b_k] \subset [t-n^6/2,t-n^6/4]$. Since on the event $\cE^c$, the process $\zeta$ remains active on the time interval $[a_k,b_k]$, and considering the definition of trustworthiness and of the Phoenix process, we know that $\zeta_{b_k} = \xi^{\un{1},a_k}_{b_k}$, and moreover, the latter random variable is distributed according to $\nu$. Hence, up to a negligible event, the probability in \eqref{e:prob4} is bounded by
$$
P\Ll[ \uptau^\nu \le n^6/2 \Rr],
$$
and thus, using \eqref{e:prob2} again,
\begin{equation}
\label{e:prob5}
P\Ll[\zeta_t = \un{0}, \cE^c \Rr] \le \frac{3 n^6}{E[\uptau]}.
\end{equation}
The conclusion now follows, combining \eqref{e:prob3} and \eqref{e:prob5}.
\end{proof}
\begin{lemma}
\label{transmission}
Let $T \in \Lambda(n,d)$ and $x \in T$. Define recursively $\gamma_0 = 0$ and, for any $i \in \N$,
$$
\gamma_{i+1} = \inf\{t \ge \gamma_i + 2 n^2 : \xi_t(x) = 1\} \quad (+ \infty \text{ if empty}).
$$
For $n$ large enough, we have
$$
P\Ll[ \gamma_{n^2/8} > n^4/2 \ | \  \xi_{n^4/2} \neq \un{0}\Rr] \le e^{-n^{2}}.
$$
\end{lemma}
\begin{proof}
In view of part (1) of Proposition~\ref{cpinterval}, for any non-empty $A \subset T$, we have
$$
P\Ll[  \exists s \le \frac{n}{\ov{c}_1} : \xi^A_s(x) = 1  \Rr] \ge \ov{c}_1.
$$
Let $\cF_i$ be the $\sigma$-field generated by $\{\xi_t, t \le \gamma_i\}$. By induction and the Markov property, we can thus show that for any $k \in \N$,
$$
P\Ll[\gamma_{i+1}-(\gamma_i+2 n^2) \ge \frac{kn}{\ov{c}_1},\  \xi_{\gamma_i+2n^2+(k-1)n/\ov{c}_1} \neq \un{0}  \ | \ \cF_i \Rr] \le (1-\ov{c}_1)^{k}.
$$
Hence,
\begin{eqnarray*}
P\Ll[ \gamma_{n^2/8} > n^4/2,  \   \xi_{n^4/2} \neq \un{0}\Rr]  & = &
P\Ll[ \sum_{i = 0}^{n^2/8-1}\gamma_{i+1}-(\gamma_i+2 n^2) > n^4/4 , \  \xi_{n^4/2} \neq \un{0}\Rr] \\
& \le &  P\Ll[\sum_{i = 0}^{n^2/8-1} B_i n /\ov{c}_1  > n^4/4 \Rr], %= P\Ll[\sum_{i = 0}^{n^2/4-1} B_i  > \ov{c}_1n^3/4 \Rr] ,
\end{eqnarray*}
where $(B_i)$ are independent geometric random variables of parameter $1-\ov{c}_1$.
For $\lambda  > 0$ small enough, we have
$$
e^{\phi(\lambda)} \stackrel{\text{(def)}}{=} E[e^{\lambda B_i}] < +\infty,
$$
and we thus obtain
$$
P\Ll[\sum_{i = 0}^{n^2/8-1} B_i  > \ov{c}_1n^3/4 \Rr] \le \exp\Ll(\phi(\lambda) n^2/8 - \lambda \ov{c}_1n^3/4\Rr),
$$
which, together with part (1) of Proposition~\ref{p:expalpha}, proves the claim.
\end{proof}
\begin{proposition}
\label{p:coupling}
For $n$ large enough, let $T \in \Lambda(n,d)$ be split into two subtrees $T_1,T_2$ as described by Lemma~\ref{lem1}. Define the process $(\td{\zeta}_t)_{t \ge 0}$ by
$$
\td{\zeta}_t = \zeta_{T_1,t} \cup \zeta_{T_2,t} \qquad (t \ge 0),
$$
where $\zeta_{T_1}$ and $\zeta_{T_2}$ are Phoenix processes defined on $T_1$ and $T_2$ respectively, using the Harris system on $T$ together with two independent families of auxiliary random variables, independent of the Harris system. One has
$$
P\Ll[ \forall t \le \uptau, \ \xi_t \ge \td{\zeta}_t \Rr] \ge 1-e^{-n^{3/2}}.
$$
\end{proposition}
\begin{proof}%[Proof of Proposition~\ref{p:coupling}]
Let $(\sigma_i)_{i \ge 1}$ be the sequence of (stopping) times when the process $\zeta_{T_1}$ becomes quiescent. We start by showing that, for any $i$,
\begin{equation}
\label{activation}
P\Ll[\xi_{\sigma_i+n^4} < \zeta_{T_1,\sigma_i+n^4}, \ \xi_{\sigma_i+n^4} \neq \un{0}\Rr] \le e^{-n^{7/4}}.
\end{equation}
For some arbitrary $x \in T_1$, consider the stopping times introduced in Lemma~\ref{transmission}, but started with $\gamma_0 = \sigma_i$, and let $N$ be the largest index satisfying $\gamma_N \le \sigma_i + n^4/2$. By Lemma~\ref{transmission}, we have
\begin{equation}
\label{e:p:coupling0}
P\Ll[N < n^2/8 ,\  \xi_{\sigma_i+n^4} \neq \un{0}\Rr] \le e^{-n^2}.
\end{equation}
Moreover, part (1) of Proposition~\ref{p:expalpha} ensures that, for any $j$,
\begin{equation}
\label{e:p:coupling}
P\Ll[\xi^{x,\gamma_j}_{T_1} \text{ survives up to time } \gamma_j + 2 n^2 \ | \ \gamma_j < +\infty \Rr] \ge \ov{c}_2,
\end{equation}
where $\xi^{x,\gamma_j}_{T_1}$ denotes the contact process restricted to $T_1$ started with $x$ occupied at time $\gamma_j$. We introduce the stopping times $\td{\gamma}_j$ to deal with the fact that $\gamma_j$ may be infinite. Let $\tdj$ we be the largest index such that $\gamma_{\tdj} \le \sigma_i + n^4/2$. We let $\td{\gamma}_j = \gamma_j$ if $j\le \tdj$, $\td{\gamma}_{\tdj+1} = \sigma_i + n^4/2 + 2 n^2$, and then recursively, $\td{\gamma}_{j+1}-\td{\gamma}_j = 2 n^2$. We have
\begin{multline}
\label{e:p:coup}
P\Ll[ \forall j \le N, \  \xi^{x,\gamma_j}_{T_1, \gamma_j + 2 n^2} = \un{0}, \ \xi_{\sigma_i+n^4} \neq \un{0}\Rr] \\
\le P\Ll[N < n^2/8,\  \xi_{\sigma_i+n^4} \neq \un{0}\Rr] + P\Ll[ \forall j \le n^2/8, \ \xi^{x,\td{\gamma}_j}_{T_1, \td{\gamma}_j + 2 n^2} = \un{0}\Rr],
\end{multline}
Since for any $j$, we have $\td{\gamma}_{j+1} \ge \td{\gamma}_j + 2 n^2$, the events indexed by $j$ appearing in the second probability on the r.h.s.\ of \eqref{e:p:coup} are independent. Using also \eqref{e:p:coupling0} and \eqref{e:p:coupling} (with $\gamma_j$ replaced by $\td{\gamma}_j$), we thus arrive at
\begin{equation}
\label{e:p:coupling1}
P\Ll[ \forall j \le N, \  \xi^{x,\gamma_j}_{T_1, \gamma_j + 2 n^2} = \un{0}, \ , \xi_{\sigma_i+n^4} \neq \un{0}\Rr] \le e^{-n^2} + (1-\ov{c}_2)^{n^2/8}.
\end{equation}
We now show that
\begin{equation}
\label{e:p:coupling2}
\exists j \le N, \  \xi^{x,\gamma_j}_{T_1, \gamma_j + 2 n^2} \neq \un{0} \ \Rightarrow \ \xi_{\sigma_i+n^4} \ge \zeta_{T_1,\sigma_i+n^4}.
\end{equation}
Indeed, in order for $\zeta_{T_1, \sigma_i+n^4}$ to be non $\un{0}$, it must be that the Harris system restricted to $T_1$ is trustworthy on $[\sigma_i,\sigma_i+n^4]$. In this case, by the definition of trustworthiness, if there exists some $j \le N$ such that $\xi^{x,\gamma_j}_{T_1, \gamma_j + 2 n^2} \neq \un{0}$, then it must be that
$$
\xi^{x,\gamma_j}_{T_1,\sigma_i+n^4} = \xi^{\un{1},\sigma_i}_{T_1,\sigma_i + n^4} \ge \zeta_{T_1,\sigma_i+n^4}
$$
(the last two being equal when $Y_{\sigma_i} = 1$, otherwise $\zeta_{T_1,\sigma_i+n^4} = \un{0}$).
Since $\xi_{\gamma_j}(x) = 1$, it is clear that $\xi_{\sigma_i+n^4} \ge \xi^{x,\gamma_j}_{T_1,\sigma_i+n^4}$, thus justifying \eqref{e:p:coupling2}. This and \eqref{e:p:coupling1} prove \eqref{activation}.

In order to conclude, we first show that $\uptau$ cannot be too large. It comes from \eqref{e:prob00} and \eqref{e:prob01} that
\begin{equation}
\label{e:p:coupling3}
P\Ll[\uptau \ge n^4 e^{C n}\Rr] \le e^{-n^{2}},
\end{equation}
where $C$ can be chosen uniformly over $T \in \Lambda(n,d)$.
If $\zeta_{T_1}$ is active at time $t$ and $\xi$ dominates $\zeta_{T_1}$ at this time, then the domination is preserved during the whole phase of activity, since $\zeta_{T_1}$ is driven by a subset of the Harris system driving the evolution of $\xi$. When $\zeta_{T_1}$ becomes quiescent, the domination is obviously preserved. As a consequence, if the domination of $\zeta_{T_1}$ by $\xi$ is broken at some time, it must be when $\zeta_{T_1}$ turns from quiescent to active. We thus have
$$
P\Ll[ \exists t \le \uptau, \ \xi_t < \zeta_{T_1,t} \Rr] = P\Ll[ \exists i : \xi_{\sigma_i+n^4} < \zeta_{T_1,\sigma_i+n^4}\text{ and } \xi_{\sigma_i+n^4} \neq \un{0} \Rr].
$$
Since $\sigma_{i+1} - \sigma_i \ge n^4$, on the event $\uptau \le n^4 e^{C n}$, there are at most $e^{C n}$ times when $\zeta_{T_1}$ turns from quiescent to active. Using \eqref{activation}, we thus obtain
$$
P\Ll[ \forall t \le \uptau, \ \xi_t \ge \zeta_{T_1,t} \Rr] \le P[\uptau \ge n^4 e^{C n}] + e^{C n} e^{-n^{7/4}}.
$$
The proposition is now proved, using \eqref{e:p:coupling3} together with the fact that
$$
P\Ll[ \exists t \le \uptau, \ \xi_t < \td{\zeta}_{t} \Rr] \le P\Ll[ \exists t \le \uptau, \ \xi_t < \zeta_{T_1,t} \Rr] + P\Ll[ \exists t \le \uptau, \ \xi_t < \zeta_{T_2,t} \Rr].
$$
\end{proof}
\begin{corollary}
\label{c:coupling}
For $n$ large enough, let $T \in \Lambda(n,d)$ be split into two subtrees $T_1,T_2$ as described by Lemma~\ref{lem1}. We have
$$
E[\uptau_T] \ge n^{-9} \  E\Ll[\uptau_{T_1}\Rr] E\Ll[\uptau_{T_2}\Rr].
$$
\end{corollary}
\begin{proof}
Let $\td{\sigma}$ be the first time when $\zeta_{T_1}$ and $\zeta_{T_2}$ are simultaneously quiescent. By Proposition~\ref{p:coupling}, for any $t \ge 0$, we have
\begin{equation}
\label{e:c0}
P[\uptau \le t] \le P[\td{\sigma} \le t] + e^{-n^{3/2}}.
\end{equation}
In view of Remark~\ref{r:randomization}, at time $\td{\sigma}$, both $\zeta_{T_1}$ and $\zeta_{T_2}$ remain quiescent for a time $n^4$ with probability at least $1/2$ (one of them just becomes quiescent at time $\td{\sigma}$, while the other stays quiescent for a time $n^4$ with probability at least $1/2$). As a consequence, for any $t \ge 0$,
$$
P\Ll[\td{\sigma} \le t\Rr] \le \frac{2}{n^4} \int_0^{t+n^4} P\Ll[\td{\zeta}_s = 0\Rr] \ \d s.
$$
Since $\zeta_{T_1}$ and $\zeta_{T_2}$ are independent, and using Lemma~\ref{l:prob0}, we thus obtain
\begin{equation}
\label{e:c1}
P[\td{\sigma} \le t] \le \frac{2}{n^4} (t+n^4) \frac{(6n^6)^2}{E[\uptau_{T_1}]E[\uptau_{T_2}]} = \frac{72 n^8(t+n^4)}{E[\uptau_{T_1}]E[\uptau_{T_2}]}.
\end{equation}
Let us now fix
$$
\td{t} = 2 \frac{{E[\uptau_{T_1}]E[\uptau_{T_2}]}}{n^9}.
$$
Since we know from part (1) of Proposition~\ref{p:expalpha} that $\td{t}$ grows faster than any power of $n$, \eqref{e:c1} gives us that for $n$ large enough,
$$
P\Ll[\td{\sigma} \le \td{t}\Rr] \le 1/4.
$$
In view of \eqref{e:c0}, we thus obtain
$$
P\Ll[\uptau \le \td{t}\Rr] \le 1/4 + e^{-n^{3/2}} \le 1/2,
$$
which implies that $E[\uptau] \ge \td{t}/2$, and thus the corollary.
\end{proof}

\begin{proof}[Proof of Theorem \ref{thm1main1}]
Let $\rho = 1+1/d$, and consider, for any $r \in \N$, the quantity
$$
V_{r}= \inf_{n \in (\rho^{r-1}/d, \rho^r]} \ \inf_{T \in \Lambda(n,d)} \frac{\log E[\uptau(T)]}{|T|}
$$
Theorem~\ref{thm1main1} will be proved if we can show that $\liminf_{r \to \infty} V_r > 0$.

Let $r$ be a positive integer, and $T$ be a tree of degree bounded by $d$ and whose size lies in $\left(\rho^{r},\rho^{r+1} \right]$.

Since $1-\rho^{-1} = 1/(d+1) < 1/d$ and in view of Lemma~\ref{lem1}, for $r$ large enough, we can split up $T$ into two subtrees $T_1$, $T_2$ such that
$$
|T_{1}|, |T_{2}| \ge |T|(1-\rho^{-1}).
$$
As a consequence,
$$
|T_1|, |T_2| \ge \rho^{r-1}/d,
$$
and also
$$
| T_{1}| \leq |T| - |T_2| \leq |T|\left( 1-(1-\rho^{-1}) \right) \le \rho^r,
$$
with the same inequality for $T_2$. Corollary~\ref{c:coupling} tells us that for $r$ large enough,
$$
E[\uptau(T)] \ge \frac{1}{|T|^{9}} E[\uptau(T_1)] \ E[\uptau(T_2)],
$$
that is to say,
$$
\log E[\uptau(T)] \ge \log E[\uptau(T_1)] + \log E[\uptau(T_2)] - \log |T|^{9}.
$$
Observing that
$$
\log E[\uptau(T_1)] + \log E[\uptau(T_2)] \ge V_r (|T_1| + |T_2|) = V_r |T|,
$$
we arrive at
\begin{equation}
\label{superadd}
\frac{\log E[\uptau(T)]}{|T|} \ge V_r - \frac{\log |T|^{9}}{|T|}.
\end{equation}
Part (1) of Proposition~\ref{p:expalpha} ensures that for $r$ large enough, one has
\begin{equation}
\label{positivvr}
V_r \ge \frac{c}{\rho^{r(1-\alpha)}}
\end{equation}
for some constant $c > 0$.
Recalling that $|T| \le \rho^{r+1}$, we thus have
$$
\frac{\log |T|^{9}}{|T|} \le \frac{V_r}{\rho^{r\alpha/2}},
$$
and \eqref{superadd} turns into
$$
\frac{\log E[\uptau(T)]}{|T|} \ge V_r\left(1-\frac{1}{\rho^{r \alpha/2}}\right),
$$
for any large enough $r$ and any tree whose size lies in $(\rho^r,\rho^{r+1}]$. If the size of the tree lies in $(\rho^r/d,\rho^r]$, then the inequality
$$
\frac{\log E[\uptau(T)]}{|T|} \ge V_r
$$
is obvious, so we arrive at
$$
V_{r+1} \ge V_r\left(1-\frac{1}{\rho^{r \alpha/2}}\right).
$$
Since $V_r > 0$ for any $r$ large enough by \eqref{positivvr}, and
$$
\prod_{r} \left(1-\frac{1}{\rho^{r \alpha/2}}\right) > 0,
$$
we have shown that $\liminf_{r \to \infty} V_r > 0$, and this finishes the proof.
\end{proof}

\begin{proof}[Proof of Theorem~\ref{thm3main3}]
Let $c > 0$ be given by Theorem~\ref{thm1main1}, and $T \in \Lambda(n,d)$. We learn from Lemma~\ref{l:attract} that
$$
P\Ll[ \uptau \le e^{cn/2} \Rr] \le \frac{e^{cn/2}}{E[\uptau]},
$$
which, by our choice of $c$, is smaller than $e^{-cn/4}$ for $n$ large enough, uniformly over $T \in \Lambda(n,d)$.
\end{proof}
%
%
%
%
%%%%%%%%%%%%%%%%%%%%%%%%%%%%%%%%%%%%%%%%%%%%%%%%%%%%%%%%%%%%%%
%%%%%%%%%%%%%%%%%%%%%%%%%%%%%%%%%%%%%%%%%%%%%%%%%%%%%%%%%%%%%%
%
%
%
\section{Discrete time growth process}
\label{s:discrete}
\setcounter{equation}{0}

For comparison purposes, it is sometimes useful to consider a discrete time analogue of the contact process; we will need to consider such a process in the next section. Though many different definitions may be proposed, we have decided on the following.

Fix $p \in (0,1)$ and let $\{I^r_{x,y}: r \in \{1, 2, \ldots\},\;x,y \in \Z,\; |x-y|\leq 1\}$ be a family of independent Bernoulli($p$) random variables. Fix $\eta_0 \in \{0,1\}^\Z$ and, for $r\geq 0$, let
$$\eta_{r+1}(x) = \mathds{1}\{\exists y: |x-y| \leq 1,\; \eta_r(y) = 1,\; I^r_{y,x} = 1\}.$$
The following is standard.

\begin{proposition}
The above process is attractive and there exists $p_{c}^{(1)} < 1$ so that for $p > p^{(1)}_{c}$ the process survives in the sense that, for any $\eta_0 \neq \underline{0}$,
$$P\left[\eta_r \neq \underline{0} \hspace{0.3cm} \forall r\right] > 0$$
and, if $\eta_0 = \underline{1}$, then $\eta_{r}$ decreases stochastically to a non zero limit.
\end{proposition}

This process generalizes to locally finite graphs, just as does the contact process. In particular it will have the self duality property and we can easily follow through the arguments of the preceding sections to arrive at
%\noindent The process has a simple duality and arguing as in Section two we have
\begin{proposition}
\label{discrete}
\noindent Let $d \geq 2$ and $p > p_c^{(1)}$. There exists $c>0$ such that
$$\inf_{T\in \Lambda(n,d)} \; P\left[\uptau_T \geq e^{cr}\right] \longrightarrow 1 \text{ as } n \to \infty.$$
\end{proposition}
\noindent(again $\uptau_T$ is the extinction time for the process on $T$ started from full occupancy).

\vspace{0.4cm}

%
%
%
%
%%%%%%%%%%%%%%%%%%%%%%%%%%%%%%%%%%%%%%%%%%%%%%%%%%%%%%%%%%%%%%
%%%%%%%%%%%%%%%%%%%%%%%%%%%%%%%%%%%%%%%%%%%%%%%%%%%%%%%%%%%%%%
%
%
%
\section{Extinction time on Newman-Strogatz-Watts random graphs}
\label{s:nsw}

Let us briefly recall the definition of the NSW random graph on $n$ vertices, $G^n = (V^n, E^n)$. We take $V^n = \{1, 2, \ldots, n\}$ and suppose given a probability $p(\cdot) $ on the positive integers greater than or equal to $3$ with the property that, for some $a > 2$ and $c_0 > 0,\; \ p(m) \sim \frac{c_0} {m^a}$.  The NSW graph $G^n$ is then generated by choosing the degrees for the $n$ vertices $d_1, d_2,  \ldots, d_n$, according to i.i.d. random variables of law $p(\cdot)$ conditioned on $\sum_{x=1}^n d_x$ being even.  Given this realization, we choose the edges by first giving each vertex $x\; d_x$ half-edges and then matching up the half-edges uniformly among all possible matchings, so that, say, a half-edge for vertex $x$ matched with a half-edge of vertex $y$ becomes an edge between $x$ and $y$.  Of course, loops and parallel edges may occur (though as noted in \cite{CD}, if $a > 3 $ the probability of nonexistence of both is bounded away from zero).

In this section we consider the contact process with small parameter $\lambda > 0$ on NSW random graphs and prove Theorem \ref{thm1cd1}. As mentioned in the Introduction, we will assume that $a > 3$.

Instead of choosing a matching for all half-edges at once, we can also match them in a sequence of steps, so that, in each step, we are free to choose one of the half-edges involved in the matching, and the other is chosen at random. To be more precise, let us introduce some terminology. A \textit{semi-graph} $g = (V^n, \mathcal{H},\mathcal{E})$ is a triple consisting of the set of vertices $V^n$, a set of half-edges $\mathcal{H}$ and a set of edges $\mathcal{E}$ (of course, if $\mathcal{H} = \varnothing$, then $g$ is a graph). The degree of a vertex in a semi-graph is the number of its half-edges plus the number of edges that are incident to it. Given two half-edges $h, h' \in \mathcal{H}$, we will denote by $h+h'$ a new edge produced by ``attaching'' $h$ and $h'$. We will now inductively define a finite sequence of semi-graphs $g_0, g_1, \ldots, g_k$ so that $g_k$ has the distribution of a NSW graph. $g_0 = (V^n, \mathcal{H}_0, \mathcal{E}_0)$ is defined with $\mathcal{E}_0 = \varnothing$ and so that each vertex $x$ has $d_x$ half-edges, where $(d_1,\ldots,d_n)$ is chosen at random as described in the previous paragraph. Assume $g_i = (V^n, \mathcal{H}_i, \mathcal{E}_i)$ is defined and has half-edges. Fix an arbitrary half-edge $h \in \mathcal{H}_i$ and randomly choose another half-edge $h'$ uniformly in $\mathcal{H}_i -\{h\}$. Then put $g_{i+1} = (V^n, \mathcal{H}_{i+1}, \mathcal{E}_{i+1})$, where $\mathcal{H}_{i+1} = \mathcal{H}_i - \{h, h'\}$ and $\mathcal{E}_{i+1} = \mathcal{E}_i \cup \{h+h'\}$. When no half-edges are left, we are done, and the graph thus obtained is a NSW random graph. Often, instead of updating the sets each time, say from $\mathcal{H}_i, \mathcal{E}_i$ to $\mathcal{H}_{i+1}, \mathcal{E}_{i+1}$ as above, we will hold the notation $g = (V^n, \mathcal{H}, \mathcal{E})$ and say (for example) that $h, h'$ are deleted from $\mathcal{H}$ and $h+h'$ is added to $\mathcal{E}$.

We will be particularly interested in vertices with degree in $\left[ S,\; 2S\right],$
where $S = M \frac{1}{\lambda^2}\log^{2} \left(\frac{1}{\lambda}\right)$ and $M$ is a large universal constant to be chosen later. We designate by $I$ the vertices of $V^n$ whose degree lies in this set.

In order to prove Theorem \ref{thm1cd1}, by attractiveness of the contact process, it is sufficient to show that given $G^n$, (with high probability as $n$ tends to infinity) there exists a subgraph on which the contact process survives for the desired amount of time. The plan is to show that for some $\delta > 0$ (that depends on $\lambda$), with high probability as $n \rightarrow \infty$, we can find a subgraph $G^{n \prime} = (V^{n\prime}, E^{n\prime})$ of $G^n$ that is a tree with certain good properties and with vertex set containing $\delta n $ vertices of $I$. Let $I'$ be the set of vertices of $V^{n\prime}$ of degree (with respect to $E^{n\prime}$) in the set $[S, 2S]$. Let us say that for $x, y \in I', \ x \stackrel{*}{\sim} y$ if $x$ and $y$ are connected to each other in $G^{n \prime}$ by a path which, apart from $x$ and $y$, contains no elements of $I'$ and which is of length less than $20a \log\left(\frac{1}{\lambda }\right)$. We wish to compare the contact process on $G^{n \prime }$ to a discrete time growth process (as in Section 4) on a tree $T$ with vertex set $I'$ and edge set $\{\{x,y\}: x \stackrel{*}{\sim} y\}$. We wish to have \vskip 3mm

\noindent \textit{Property A:} $T$ is a tree of degree bounded by 4.\vskip 3mm

\noindent \textit{Property B:} every element of $I'$ has $\frac{S}{4}$ neighbors of degree 1.\vskip 3mm

Property B ensures that around each site in $I$, the infection persists for a long time. This guarantees that our discrete time growth process has infection rate as large as desired. Together with Property A, this allows us to apply Proposition \ref{discrete} to the growth process and conclude that its extinction time is very large. We then conclude that the extinction time of the contact process on $G^{n\prime}$ is also very large.

We will find the subgraph $G^{n\prime}$ with the aid of an algorithm whose starting point will be the semi-graph $g_0$ defined above. Before we present the algorithm let us make some remarks about the random degree sequence $d_1, \ldots, d_n$.

Let $\mu = \sum_{m=1}^\infty m\cdot p(m)$. Let us remark that, if the degrees are given by $d_1, \ldots, d_n$ and we choose a half-edge uniformly at random in $g_0$, then the probability that the corresponding vertex has degree $m$ is $$\frac{m \cdot |x:d_x = m|}{\sum_x d_x} \to \frac{m \cdot p(m)}{\mu}\text{ as } n \to \infty.$$ The probability $q(m) = m\cdot p(m)/\mu$ is called the \textit{size biased distribution}. By our assumption that $p(m) \sim \frac{c_0}{m^a}$, it follows that $q(m) \sim \frac{c_1}{m^{a-1}}$, where $c_1 = \frac{c_0}{\mu}$. If $x$ is large enough, it can be easily verified by comparison with an integral that
$$\frac{c_1}{2(a-2)}\cdot x^{-(a-2)} < q([x, 2x]) < \frac{2 c_1}{a-2}\cdot x^{-(a-2)}.$$
We will also need the following facts, whose proofs are omitted.

\begin{lemma}
\label{lem5degs}
For any small enough $\lambda > 0$, there exists $\epsilon > 0$ such that, with probability tending to 1 as $n \to \infty$, for any $A \subset V^n$ with $|A| \leq \epsilon n$ we have\medskip\\
$(i.)\;\displaystyle{\frac{c_0}{2(a-1)S^{a-1}} < \frac{|I\cap A^c|}{n} < \frac{2c_0}{(a-1)S^{a-1}} };$\medskip\\
$(ii.)\;\displaystyle{\frac{\sum_{x \in A}\;d_x}{\sum_{x \in V^n}\; d_x} < \frac{1}{8}};$\medskip\\
$(iii.)\;\displaystyle{\frac{c_1}{2(a-2)S^{a-2}}<\frac{\sum_{x\in I \cap A^c}\; d_x}{\sum_{x \in V^n}\; d_x} < \frac{2c_1}{(a-2)S^{a-2}};}$\medskip\\
$(iv.)\;\displaystyle{\frac{\mu}{2}<\frac{\sum_{x \in A^c} \;d_x}{n} < 2\mu}.$
\end{lemma}

The hypothesis that $\lambda$ is small is not problematic to us because clearly it is sufficient to prove Theorem \ref{thm1cd1} for $\lambda$ small enough. In what follows, $\lambda$ is fixed and $\epsilon$ is taken corresponding to $\lambda$ as in the lemma. We will often assume that $\lambda$ is small enough, and also that $n$ is large enough, for other desired properties to hold.  We will say that a degree sequence $d_1, \ldots, d_n$ is \textit{robust} if it satisfies $(i.), (ii.), (iii.)$ and $(iv.)$. Our algorithm proceeds by matching half-edges, as described in the beginning of this section. We will thus have to deal with the set of half-edges after some matchings have been made, and that is why the robustness property will come into play.

Other than match half-edges, the algorithm also writes labels on edges and vertices. Edges are labeled either $\mathsf{marked}$ or $\mathsf{unmarked}$; the former are included in $G^{n\prime}$ and the latter are not. Vertex labels serve to guide the order of the matchings. The possible vertex labels are: $\mathsf{unidentified}$, $\mathsf{preactive}$, $\mathsf{active}$ and $\mathsf{read}$. An $\mathsf{unidentified}$ vertex is one that has not yet been ``seen'' by the algorithm, that is, none of its half-edges has been matched yet. If a vertex has any label different from $\mathsf{unidentified}$, then it is said to be identified. The labels $\mathsf{preactive}$ and $\mathsf{active}$ can only be associated to vertices in $I$, and at most one vertex will be $\mathsf{active}$ at a given time.

The algorithm repeatedly follows a subroutine called a \textit{pass}. Between two passes, there will be no $\mathsf{active}$ vertices. When a new pass starts, it typically takes a $\mathsf{preactive}$ vertex $\bar x$, turns it into $\mathsf{active}$ and successively explores the graph around $\bar x$ (by performing matchings) until certain conditions are satisfied; then, it labels every vertex that was touched as $\mathsf{read}$ except for the vertices of $I$ that were found; these are labeled $\mathsf{preactive}$ and are activated by future passes. %Each vertex $x$ evolves along the following line:\medskip\\
%$\bullet\;$ if $x \in  I:\quad$ \begin{tabular}{|l} $\mathsf{unidentified}$ $\rightarrow$ $\mathsf{preactive}$ $\rightarrow$ $\mathsf{active}$ $\rightarrow$ %$\mathsf{read}$  or\\$\mathsf{unidentified}$ $\rightarrow$ $\mathsf{preactive}$ $\rightarrow$ $\mathsf{read}$\medskip\\
%$\bullet\;$ if $x \notin I:\quad$ $\mathsf{unidentified}$ $\rightarrow$ $\mathsf{read}$ \medskip

A labeled semi-graph $g=(V^n, \mathcal{H}, \mathcal{E}, \{\ell_x\}_{x\in V^n},\{\ell_e\}_{e\in \mathcal{E}}, \prec)$ is a semi-graph with a label $\ell_x$ attached to each vertex $x$, a label $\ell_e$ associated to each edge $e$ and a total order $\prec$ on the set of $\mathsf{preactive}$ vertices. It is worth remarking that since a pass only does matchings and relabeling, it does not change the degree of any vertex. In particular, the definition of the set $I$ does not change. Let us now define the pass. Obviously, whenever there is an instruction to give a vertex a label, this label replaces the former label of that vertex. \vskip 2mm

\noindent\rule[0.5ex]{\linewidth}{1pt}
\textbf{The pass\\ Input:} $g = (V^n, \mathcal{H}, \mathcal{E}, \{\ell_x\},\{\ell_e\}, \prec)$ with at least one vertex of $I$ $\mathsf{preactive}$ or $\mathsf{unidentified}$.
\smallskip \\
\begin{tabular}{p{0.45cm} p{11.6cm}}
\textbf{(S1)}& Let $\bar x$ be the $\mathsf{preactive}$ vertex of highest order; if there are no $\mathsf{preactive}$ vertices, let $\bar x$ be an arbitrary $\mathsf{unidentified}$ site of $I$.\\
&$\bullet\;$ If $\bar x$ has less than $\frac{S}{2}$ half-edges (which can only happen if it is $\mathsf{preactive}$), label it $\mathsf{read}$; the pass is then stopped in status $\mathsf{B}_1$.\\&$\bullet\;$ Otherwise, label $\bar x$ $\mathsf{active}$ and proceed to (S2).\end{tabular}\smallskip \\
\begin{tabular}{p{0.45cm} p{11.6cm}}
\textbf{(S2)}& Define the set $\mathcal{H}^*$ of \textit{relevant half-edges} of the pass as the set of half-edges attached to the $\mathsf{active}$ vertex. Endow $\mathcal{H}^*$ with a total order $\prec^*$ chosen arbitrarily. Also let $\bar C = 0$; this will be a counting variable whose value will be progressively incremented. Proceed to (S3).\end{tabular}\smallskip\\
\begin{tabular}{p{0.45cm} p{11.6cm}}
\textbf{(S3)} &Let $h$ be the half-edge of highest order in $\mathcal{H}^*$. Choose another half-edge $h'$ uniformly at random in $\mathcal{H} - \{h\}$ and let $v'$ be the vertex of $h'$. Delete $h, h'$ from all sets that contain them ($h$ from $\mathcal{H}$ and $\mathcal{H}^*$, $h'$ from $\mathcal{H}$ and possibly $\mathcal{H}^*$) and add $h+h'$ to $\mathcal{E}$; its label is given as follows:\\
&$\bullet\;$ If $v'$ is identified, label $h+h'$ $\mathsf{unmarked}$.\\
&$\bullet\;$ If $v'$ is $\mathsf{unidentified}$ and not in $I$, label $h+h'$ $\mathsf{marked}$. Also label $v'$ $\mathsf{read}$ and add its half-edges to $\mathcal{H}^*$ (note that at this point $h'$ is no longer a half-edge of $v'$) so that they have arbitrary order among themselves but lower order than all half-edges previously in $\mathcal{H}^*$. \\
&$\bullet\;$ If $v'$ is $\mathsf{unidentified}$ and in $I$, and if $\bar C < 3$, label $h+h'$ $\mathsf{marked}$, label $v'$ $\mathsf{preactive}$, assign it the lowest order in the set of $\mathsf{preactive}$ vertices and add 1 to $\bar C$.\\
&$\bullet\;$ If $v'$ is $\mathsf{unidentified}$ and in $I$, and if $\bar C \geq 3$, label $h+h'$ $\mathsf{unmarked}$ and label $v'$ $\mathsf{read}$.\\
&Proceed to (S4).\end{tabular}\smallskip\\
\begin{tabular}{p{0.45cm} p{11.6cm}}
\textbf{(S4)} &$\bullet\;$ If $\bar x$ still has half-edges, go to (S3).\\&$\bullet\;$ If the last half-edge of $\bar x$ has been deleted in the previous step and now there are less than $\frac{S}{4}$ $\mathsf{marked}$ edges incident to $\bar x$, label $\bar x$ and all vertices that have been identified in the pass (including the $\mathsf{preactive}$ ones) $\mathsf{read}$. The pass is stopped in status $\mathsf{B}_2$.\\&$\bullet\;$ Otherwise go to (S5).\end{tabular}
\smallskip\\
\begin{tabular}{p{0.45cm} p{11.6cm}}
\textbf{(S5)} &$\bullet\;$ If $\bar C \geq 3$, label $\bar x$ $\mathsf{read}$ and end the pass in status $\mathsf{G}$.\\&$\bullet\;$ Otherwise go to (S6).
\end{tabular}\smallskip\\
\begin{tabular}{p{0.45cm} p{11.6cm}}
\textbf{(S6)} &$\bullet\;$ If (a) more than $\left(\frac{1}{\lambda} \right)^{2a-3}$ vertices have been identified in the pass, or (b) a path of length  $20a\log\left(\frac{1}{\lambda}\right)$ may be formed with $\mathsf{marked}$ edges constructed in the pass, or (c) $\mathcal{H}^*$ is empty, then label $\bar x$ $\mathsf{read}$ and end the pass in status $\mathsf{B}_3$.
\\&$\bullet\;$ Otherwise go to (S3).\end{tabular}
\smallskip\\
\textbf{Output:} updated labeled semi-graph, status.\\
\noindent\rule[0.5ex]{\linewidth}{1pt}

Let us explain in words what happens when a pass ends in status $\mathsf{G}$. It first activates the $\mathsf{preactive}$ vertex $\bar x$ of highest order, then starts identifying the neighbors of $\bar x$; when they are all identified, it starts identifying the vertices at distance 2 from $\bar x$, and so on, until it has found three new vertices of $I$, at which point it stops. The ``bad'' outcomes $\mathsf{B}_1$,$\mathsf{B}_2$ and $\mathsf{B}_3$ are included to guarantee that $G^{n\prime}$ has the desired properties mentioned earlier and that the algorithm can successfully continue. $\mathsf{B}_1$ and $\mathsf{B}_2$ are necessary to ensure that the vertices of $G^{n\prime}$ that will be the focal points for the comparison growth process all have large degree. $\mathsf{B}_3$ is necessary to ensure that the focal points are not very far from each other and also that the pass does not delete too many half-edges, thus exploring too much of the graph.

We wish the pass to return the status $\mathsf{G}$ ; the following lemma addresses this.

\begin{lemma}
\label{lem5goodpass}
Assume that the degree sequence is robust, $g$ has less than $\frac{\epsilon}{2} n$ identified vertices before the pass starts and, when the pass defines $\bar x$, this vertex has more than $\frac{S}{2}$ half-edges. Then, the pass ends in status $\mathsf{G}$ with probability larger than $\frac{9}{10}$.
\end{lemma}
\begin{proof}
Start noticing that the pass identifies at most $(1/\lambda)^{2a-3}$ vertices and this is much less than $\frac{\epsilon}{2}n$ if $n$ is large. So, at a moment immediately before the pass chooses a half-edge at random, there are less than $\epsilon n$ identified vertices; let $A$ in the definition of robustness be this set of identified vertices. The chosen half-edge then has probability:\\
(1) larger than $\frac{7}{8}$ of belonging to an $\mathsf{unidentified}$ vertex;\\
(2) larger than $\frac{c_1}{2(a-2)S^{a-2}}$ of belonging to an $\mathsf{unidentified}$ vertex that is in $I$;\\
(3) larger than $\frac{3}{4}$ of belonging to an $\mathsf{unidentified}$ vertex that is not in $I$.

By hypothesis, the pass does not end in status $\mathsf{B}_1$. For it to end in status $\mathsf{B}_2$, at least half of the more than $\frac{S}{2}$ half-edges initially present in $\mathcal{H}^*$ must be matched to half-edges of previously identified vertices. By (1) this has probability less than $P[\mathsf{Bin}(S/2,1/8)>S/4]$, which is less than $\frac{1}{40}$ if $\lambda$ is small (and hence $S$ is large). Likewise, we can show using (2) that the probability of the pass ending because of case (a) in (S6) is less than $\frac{1}{40}$.

Let us now show that the same holds for (b) in (S6). For $k \geq 1$, let $s_k$ be the set of vertices at distance $k$ from $\bar x$ that are not in $I$ and that the pass identifies; also let $s_0 = \{\bar x\}$. Since every vertex has degree 3 or more, there will be at least $2|s_k|$ half-edges of vertices of $s_k$ for the pass to match (unless it halts before). Define the event $$A_k =\left\{\begin{array}{c} \text{the pass deletes all half-edges of vertices of $s_k$; of these, }\\\text{less than $\frac{5}{8}$ are matched to half-edges of vertices not in I}\\\text{ that were (at the time of matching) $\mathsf{unidentified}$} \end{array}\right\},\quad k \geq 0.$$
We have $\P[A_0] \leq P\left[\mathsf{Bin}(S/2, 3/4) < (S/2) \cdot (5/8)\right].$ Given that $A_1, \ldots, A_k$ have not occurred and the pass reaches distance $k+1$ from $\bar x$, the probability of $A_{k+1}$ is less than
$$P\left[\mathsf{Bin}\left(\frac{S}{2}\left(\frac{5}{8}\right)^{k+1} 2^{k},\; \frac{3}{4}\right) <  \left(\frac{S}{2}\left(\frac{5}{8}\right)^{k+1} 2^{k} \right) \cdot \frac{5}{8}\right].$$
Letting $K = \frac{2a}{\log(5/4)}\log\left(\frac{1}{\lambda}\right) < 20a\log\left(\frac{1}{\lambda}\right)$, the above estimates show that $\displaystyle{\P\left[\cup_{k=0}^K\; A_k\right]}$ vanishes as $\lambda \to 0$. Now, assume that $A_0, \ldots, A_K$ have not occurred and the pass reaches level $K+1$. The probability that less than 3 $\mathsf{unidentified}$ vertices of $I$ are discovered in the matching of half-edges from $s_{K+1}$ is then less than
\begin{equation}P\left[\mathsf{Bin}\left(\frac{S}{2}\left(\frac{5}{8}\right)^{K+1} 2^{K},\; \frac{c_1}{2(a-2)S^{a-2}} \right) < 3\right].\label{eqn5Bin3}\end{equation}
Note that
$$\frac{S}{2}\left(\frac{5}{8}\right)^{K+1} 2^{K} \geq \frac{S}{4} \left(\frac{1}{\lambda}\right)^{\frac{2a}{\log(5/4)}\cdot \log(5/4)} = \frac{S}{4}\left(\frac{1}{\lambda}\right)^{2a}$$
and
$$\frac{c_1}{2(a-2)S^{a-2}} = \frac{c_1 \lambda^{2(a-2)}}{M^{a-2}\log^{2(a-2)}(1/\lambda)},$$
so the probability in (\ref{eqn5Bin3}) is very small if $\lambda$ is small. Putting these facts together, we get the desired result.

The probability of (c) in (S6) occurring is similarly shown to be less than $\frac{1}{40}$, and this concludes the proof.
\end{proof}

From now on, we will assume that the degree sequence is robust. With the definition of the pass at hand, we are now ready to explain the full algorithm. From the degree sequence $d_1, \ldots, d_n$, we construct our initial labeled semi-graph $g$ containing no edges and so that each vertex $x$ is $\mathsf{unidentified}$ and has $d_i$ half-edges. We then run $\epsilon'n$ successive passes, where $\epsilon'=\frac{\epsilon \lambda^{2a-3}}{2}$. Since each pass identifies at most $\lambda^{-(2a-3)}$ vertices, we see that at the beginning of each pass, less than $\frac{\epsilon}{2}n$ vertices will be identified, so the hypotheses of Lemma \ref{lem5goodpass} will hold. Also let $\delta' = \frac{\epsilon'}{2}$

For $1 \leq i < \epsilon' n$, define
$$\begin{aligned}
&W_i = \text{Number of $\mathsf{preactive}$ vertices before pass $i$},\\
&X_i = W_{i+1} - W_{i},\\
&Y_i = \mathds{1}_{\{\text{Pass $i$ ends in status $\mathsf{B}_1$}\}}.
\end{aligned}$$
The possible values for $X_i$ are $-1, 0, 1, 2$. If $Y_i = 1$, then $X_i = -1$. By the previous lemma, for any $x_1, \ldots, x_{i-1}, y_1, \ldots, y_{i-1}$ we have
\begin{equation}\label{eqn5Xaflem}\P\left[\;X_i = 2 \;|\;\{X_j\}_{j=1}^{i-1} = \{x_j\}_{j=1}^{i-1},\; \{Y_j\}_{j=1}^{i-1}=\{y_j\}_{j=1}^{i-1},\;Y_i = 0\;\right] > 9/10.\end{equation}
Let us now exclude the possibility that many passes end in status $\mathsf{B}_1$.
\begin{lemma}
\label{lem5manyB1}
$$\mathbb{P}\left[\sum_{i=1}^{\lfloor \epsilon'n\rfloor} Y_i > \frac{1}{10}\lfloor \epsilon'n \rfloor\right] \xrightarrow{n \to \infty} 0.$$
\end{lemma}
\begin{proof}We start remarking that, for $\{Y_i = 1\}$ to occur, there must exist a vertex $x \in I$  such that\\
$\bullet\;$ $x$ is identified before pass $i$;\\
$\bullet\;$ from the moment $x$ is identified to the beginning of pass $i$, more than $S/2$ half-edges of $x$ are chosen for matchings;\\
$\bullet\;$ $x$ is the $\mathsf{preactive}$ vertex of highest order when pass $i$ starts.

Let $h_1, \ldots, h_N$ be the sequence of half-edges chosen at random by the algorithm. As explained above, we have $N \leq \epsilon n$. By $(iii.)$ of Lemma \ref{lem5degs}, regardless of what happened before $h_j$ is chosen, the probability that $h_j$ belongs to a vertex of $I$ is less than $\frac{2c_1}{(a-2)S^{a-2}}$. On the other hand, for $\{\sum Y_i > (1/10)\lfloor \epsilon'n \rfloor\}$ to occur, more than $\frac{1}{10} \lfloor \epsilon'n\rfloor \frac{S}{2}$ half-edges of vertices of $I$ must be chosen. The probability of this is less than
$$P\left[\mathsf{Bin}\left(\lfloor \epsilon n \rfloor,\;\frac{2c_1}{(a-2)S^{a-2}} \right) > \frac{1}{10} \lfloor\epsilon'n\rfloor \frac{S}{2}\right].$$
By Markov's Inequality, this is less than
$$\begin{aligned} \frac{\lfloor \epsilon n \rfloor \cdot \frac{2c_1}{(a-2)S^{a-2}}}{\frac{1}{10} \lfloor\epsilon'n\rfloor \frac{S}{2}} &= C\frac{\epsilon}{\epsilon'} \frac{1}{S^{a-1}} = C \frac{1}{\lambda^{2a-3}} \frac{\lambda^{2(a-1)}}{M^{a-1}\log^{2(a-1)}(1/\lambda)} = C' \frac{\lambda}{\log^{2(a-1)}(1/\lambda)},
\end{aligned}$$
where $C, C'$ are constants that do not depend on $\lambda$ or $n$. The above can be made as small as desired by taking $\lambda$ small.
\end{proof}

\begin{proposition}$\displaystyle{\P\left[W_{\lfloor \epsilon'n \rfloor} > \delta' n\right] \xrightarrow{n\to\infty} 1}.$
\end{proposition}
\begin{proof}
We start giving a random mapping representation of the random variables $X_1, \ldots, X_{\lfloor\epsilon'n\rfloor}, Y_1, \ldots, Y_{\lfloor \epsilon'n \rfloor}$.
Given sequences $\{x_j\}_{j=1}^{i-1},\; \{y_j\}_{j=1}^{i}$ and $s \in (0,1)$, let
$$\begin{aligned}
&\upphi\left(s,\{x_j\}_{j=1}^{i-1},\{y_j\}_{j=1}^i\right) = m \text{ if }\\
&\P\left[X_i \leq m-1\;|\;\{X_j\}_{j=1}^{i-1} = \{x_j\}_{j=1}^{i-1},\;\{Y_j\}_{j=1}^{i} = \{y_j\}_{j=1}^{i}\right] \\&\qquad \qquad< s \leq \P\left[X_i \leq m\;|\;\{X_j\}_{j=1}^{i-1} = \{x_j\}_{j=1}^{i-1},\;\{Y_j\}_{j=1}^{i} = \{y_j\}_{j=1}^{i}\right]
\end{aligned}$$
Likewise, let
$$\begin{aligned}
&\uppsi\left(s,\{x_j\}_{j=1}^{i-1},\{y_j\}_{j=1}^{i-1}\right)\\&\qquad\qquad=\left|\begin{array}{ll} 0 &\text{ if } s \leq \P\left[Y_i = 0\;|\;\{X_j\}_{j=1}^{i-1} = \{x_j\}_{j=1}^{i-1},\;\{Y_j\}_{j=1}^{i-1} = \{y_j\}_{j=1}^{i-1}\right]\\1&\text{otherwise.}\end{array}\right.
\end{aligned}$$
(when we write only $\upphi(s),\; \uppsi(s)$, we mean the functions above for $X_1$ and $Y_1$, with no conditioning in the probabilities that define them).
Let $U_1, U_2, \ldots,\;V_1, V_2, \ldots$ be independent random variables with the uniform distribution on $(0,1)$. Set $X_1' =\upphi(U_1),\; Y_1' = \uppsi(V_1)$ and recursively define, for $1 < i < \epsilon'n$,
$$\begin{aligned}
&Y_{i+1}' = \uppsi\left(V_{i+1}, \{X_j'\}_{j=1}^{i}, \{Y_j'\}_{j=1}^i\right);\\
&X_{i+1}' = \upphi\left(U_{i+1},\{X_j'\}_{j=1}^{i}, \{Y_j'\}_{j=1}^{i+1}\right).
\end{aligned}$$
Now, clearly $\{X_i', Y_i'\}_{i=1}^{\lfloor \epsilon'n\rfloor}$ has the same distribution as $\{X_i, Y_i\}_{i=1}^{\lfloor \epsilon'n\rfloor}$. By (\ref{eqn5Xaflem}), we have $\{Y_i' = 0,\;X_i' \neq 2 \} \subset \{U_i \leq \frac{1}{10}\}$. We can now estimate
$$\begin{aligned}
\P\left[W_{\lfloor \epsilon'n\rfloor} < \frac{\epsilon'}{2}n\right] &\leq \P\left[\sum_i Y_i > \frac{1}{10}\epsilon'n\right] + P\left[\left|\{i: Y_i =0, X_i \neq 2\} \right| > \frac{1}{5}\epsilon'n\right]\\
&\leq \P\left[\sum_i Y_i > \frac{1}{10}\epsilon'n\right] + P\left[\left|\left\{i \leq \epsilon'n: U_i \leq \frac{1}{10}\right\}\right| > \frac{1}{5}\epsilon'n\right].
\end{aligned}$$
The first of these probabilities vanishes by Lemma \ref{lem5manyB1}, and the second by the Law of Large Numbers.
\end{proof}

We will define our subgraph $G^{n \prime}$ only on the event $\{W_{\lfloor \epsilon'n \rfloor > \delta'n}\}$. Let
$$\begin{aligned}
&i_0 = \sup\{i: W_i = 0\};\\
&V^{n\prime} = \text{vertices that have been identified in passes $i_0, \ldots, \lfloor \epsilon'n\rfloor$};\\
&E^{n\prime} = \text{edges that have been constructed by passes $i_0, \ldots, \lfloor \epsilon'n\rfloor$ and are }\mathsf{marked};\\
&G^{n\prime} = (V^{n\prime}, E^{n\prime});\\
&\deg'(x) = |\{e \in E^{n\prime}: x \in e\}|;\\
&I' = \{x \in V^{n\prime}: x \text{ has been activated by a pass after $i_0$ and } \deg'(x) \geq S/4\}.
\end{aligned}$$
Let $\delta = \delta'/2$ and note that $|I'| > \delta n$. This follows from the fact that, in the sequence $X_{i_0},\ldots, X_{\lfloor \epsilon'n\rfloor}$, for every $i$ such that $X_i \neq -1$, the vertex that was activated in pass $i$ must be in $I'$. Since there are at least $\frac{\delta'}{2}n$ such $i$'s, we get $|I'| > \delta n$.

As we have already mentioned, for $x, y \in I'$ we put $x \stackrel{*}{\sim} y$ if $x$ and $y$ are connected by a path (in $G^{n\prime}$) that contains no other elements of $I'$. If it exists, this path is necessarily unique and has length less than $20a\log(1/\lambda)$. We then define $T$ as the tree with vertex set $I'$ and edge set $\{\{x,y\}: x,y\in I',\;x\stackrel{*}{\sim}y\}$.

Given a vertex $x \in G^{n\prime}$, we will denote by $S(x)$ the set containing $x$ and its neighbors (in $G^{n\prime}$). If $x, y \in I',\; x\stackrel{*}{\sim} y$, let $b(x,y)$ be the set of vertices of $G^{n\prime}$ in the unique path from $x$ to $y$.

Our goal now is to use Proposition $\ref{discrete}$ to show Theorem $\ref{thm1cd1}$. To this end, we will couple the contact process on $G^{n \prime}$ (starting from full occupancy) and a growth process on $T$ (again starting from full occupancy). This comes down to a coupling between the Harris system on $G^{n \prime}$ and the Bernoulli random variables used to define the growth process.

We suppose given the Harris system on $G^n$ which we will regard as a Harris system on $G^{n\prime}$ by ignoring non relevant Poisson processes. We will consider the process on time intervals of size $\kappa = e^{30a\log(1/\lambda)}$; this scale is chosen because it is large enough for an infection from a site $x \in I'$ to reach $y \in I'$ with $x \stackrel{*}{\sim} y$ but smaller than the extinction time for the process restricted to $S(x)$. The following lemma and proposition will make this precise.

Given a set of vertices $U$ in a graph $\Gamma$ and $\xi \in \{0,1\}^U$, we will say that $U$ is \textit{infested} in $\xi$ if $|\{x\in U: \xi(x) = 1\}| \geq \frac{\lambda}{20}|U|$. In \cite{MVY} the following is proved.

\begin{lemma}
\label{lem6infested} Given $\lambda > 0$, there exist $\bar c_6$ and $N_0$ such that the following holds. Let $\Gamma$ be a star graph consisting of one vertex $x$ of degree $\frac{N}{\lambda^2}$, where $N \geq N_0$, and all other vertices of degree 1. Then, for the contact process with parameter $\lambda$ on $\Gamma$,\medskip\\
$(i.)\; \displaystyle{P_{\Gamma, \lambda}\left[\Gamma \text{ is infested in } \xi_1 \left| \xi_0 = \{x\}\right.\right] > 1/2}$;\medskip\\
$(ii.)\; \displaystyle{P_{\Gamma, \lambda}\left[\Gamma \text{ is infested in } \xi_{e^{\bar c_6 N}} \left| \Gamma \text{ is infested in } \xi_0\right.\right] > 1 - e^{-\bar c_6 N}}$.
\end{lemma}

In the case of the star graph given by a site $x \in I'$ and its neighbors in $G^{n\prime}$, the $N$ of the above lemma is equal to $\lambda^2 \deg'(x) = M\log^2\left(\frac{1}{\lambda}\right)$. The extinction time for the contact process restricted to $S(x)$ and started from full occupancy will then be with high probability larger than $e^{\bar c_6 M \log^2(1/\lambda)} = \left( \frac{1}{\lambda}\right)^{\bar c_6 M\log(1/\lambda)} > \left(\frac{1}{\lambda}\right)^{21a\log(1/\lambda)}$ as long as $M > \frac{21a}{\bar c_6}$. Now, if $x, y \in I'$ and $x\stackrel{*}{\sim} y$, the probability that an infection in $S(x)$ is transmitted along $b(x, y)$, reaches $y$ within time $20a\log\left(\frac{1}{\lambda}\right)$ and then infests $S(y)$ within time 1 is larger than $\frac{1}{2}\left(\lambda/(1+\lambda)\right)^{20a\log(1/\lambda)}$. If $S(x)$ holds the infection for $(1/\lambda)^{21a\log(1/\lambda)}$ units of time, there will be $\displaystyle{\frac{(1/\lambda)^{21a\log(1/\lambda)}}{20a\log(1/\lambda)+1}}$ chances for such a transmission to occur. Comparing the number of chances with the probability of a transmission, we see that a transmission will occur with very high probability. These considerations lead to

\begin{proposition}
For any $\sigma > 0$, $M$ can be chosen large enough so that the following holds. Assume that $x, y \in I',\; x \stackrel{*}{\sim} y$ and, in $\xi_0,\;S(x)$ is infested. Let $(\xi_t')$ denote the process restricted to $S(x) \cup S(y)\;\cup b(x,y)$. Then, with probability larger than $1-\sigma$, both $S(x)$ and $S(y)$ are infested in $\xi_\kappa'$.\end{proposition}

Let $r \in \N,\;x, y \in I'$ with $x \stackrel{*}{\sim} y$ and $(\xi_t)$ be the contact process on $G^{n\prime}$ started from full occupancy. Put $I^r_{x,y} = 1$ if one of the following holds:\\
$\bullet\;$ $S(x)$ is infested in $\xi_{\kappa r}$ and
$$\left|\left\{\begin{array}{c} z \in S(y): \exists w\in S(x): \xi_{\kappa r}(w) = 1,\\(w,\kappa r) \leftrightarrow (z,\kappa(r+1)) \text{ inside } b(x,y) \end{array}\right\}\right| > \frac{\lambda}{20}|S(y)|;$$
$\bullet\;$ $S(x)$ is not infested in $\xi_{\kappa r}$.\\
Otherwise put $I^r_{x,y} = 0$. The second condition above is just present to guarantee that $I^r_{x,y} = 1$ with high probability regardless of $\xi_{\kappa r}$. As will soon be seen, this artificial assignment will not be problematic. Put $I^r_{x,x} = 1$ if one of the following holds:\\
$\bullet\;$ $S(x)$ is infested in $\xi_{\kappa r}$ and $$\left|\left\{\begin{array}{c}z \in S(x): \exists w \in S(x): \xi_{\kappa r}(w) = 1,\\ (w, \kappa r) \leftrightarrow (z, \kappa(r+1)) \text{ inside }S(x)\end{array}\right\}\right| > \frac{\lambda}{20}|S(x)|;$$
$\bullet\;$ $S(x)$ is not infested in $\xi_{\kappa r}$.\\
Otherwise put $I^r_{x,x} = 0$

Let $\eta_0 \equiv 1$ and, for $r \geq 0$, $$\eta_{r+1}(x) = \mathds{1}\left\{\begin{array}{c}\eta_r(x) = 1 \text { and } I^r_{x,x} = 1 \text{ or, for some }\\\text{$y$ with $x\stackrel{*}{\sim}y,\;\eta_r(y) = 1$  and $I^r_{y,x} = 1$.} \end{array}\right\}.$$
Notice that, if a sequence $x_1, x_2, \ldots, x_R$ in $I'$ is such that, for each $r$, either $x_r = x_{r+1}$ and $I^r_{x_r, x_r} = 1$ or $x_r \stackrel{*}{\sim} x_{r+1}$ and $I^r_{x_r, x_{r+1}} = 1$, then we will have $\eta_R(x_R) = 1$ and $S(x_R)$ will be infested in $\xi_{\kappa R}$.

Now, using a result of Liggett, Schonmann Stacey \cite{LSchS} (see also Theorem B26 in \cite{lig99}), given $p \in (0,1)$ we can choose $M$ large enough that the measure of the field $\{\{I^r_{x,x}\},\{I^r_{x,y}\}\}$ stochastically dominates i.i.d. Bernoulli($p$) random variables. We then have
\begin{corollary}
For any $p > p_c(1)$, if $M$ is large enough, then $\{\eta_r\}$ dominates a growth process on $T$ defined from i.i.d. Bernoulli($p$) random variables.
\end{corollary}
\noindent This, the fact that $|I'| \geq \delta n$ and Proposition $\ref{discrete}$ give Theorem $\ref{thm1cd1}$.


\begin{thebibliography}{99}

\bibitem[BBCS05]{BBCS} N.\ Berger, C.\ Borgs, J.T.\ Chayes, A. Saberi. On the spread of viruses on the internet. \emph{Proceedings of the sixteenth annual ACM-SIAM symposium on discrete algorithms}, 301-310 (2005).

\bibitem[CGOV84]{eulalia} M. Cassandro, A. Galves, E. Olivieri, M. Vares. Metastable behavior of stochastic dynamics: a pathwise approach. \emph{J. Statist. Phys.} \textbf{35} (5-6), 603-634 (1984).

\bibitem[CD09]{CD} S.\ Chatterjee, R.\ Durrett. Contact process on random graphs with degree power law distribution have critical value zero. \emph{Ann. Probab.} \textbf{37}, 2332-2356 (2009).

\bibitem[DZ]{dembzeit} A.\ Dembo, O.\ Zeitouni. \textit{Large deviations techniques and applications} (2nd ed.). Applications of mathematics \textbf{38}, Springer (1998).

\bibitem[Du1]{durprob} R.\ Durrett, \textit{Probability: theory and examples} (4th ed.). Cambridge university press (2010).

\bibitem[Du2]{Dur} R.\ Durrett. \textit{Random graph dynamics}. Cambridge university press (2010).

\bibitem[DL88]{durliu} R.\ Durrett, X.\ Liu. The contact process on a finite set. \emph{Ann. Probab.} \textbf{16}, 1158-1173 (1988).

\bibitem[DS88]{dursc} R.\ Durrett, R.H.\ Schonmann. The contact process on a finite set II. \emph{Ann. Probab.} \textbf{16}, 1570-1583 (1988).

\bibitem[DST89]{dursctan} R.\ Durrett, R.H.\ Schonmann, N.\ Tanaka. The contact process on a finite set III: the
critical case. \emph{Ann. Probab.} \textbf{17}, 1303-1321 (1989).

\bibitem[Li1]{lig85} T. Liggett, \emph{Interacting particle systems}. Grundlehren der mathematischen Wissenschaften \textbf{276}, Springer (1985).

\bibitem[Li2]{lig99} T. Liggett, \emph{Stochastic interacting systems: contact, voter and exclusion processes}. Grund\-lehren der mathematischen Wissenschaften \textbf{324}, Springer (1999).

\bibitem[LSS97]{LSchS}  T.\ Liggett, R.\ Schonmann, A.\ Stacey. Domination by product measures. \emph{Ann. Probab.} \textbf{25}, 71-95 (1997).

\bibitem[Mo93]{tommeta} T.\ Mountford. A metastable result for the finite multidimensional contact process, \emph{Canad. Math. Bull.} \textbf{36} (2), 216-226 (1993).

\bibitem[Mo99]{tomexp} T.\ Mountford. Existence of a constant for finite system extinction. \emph{J. Statist. Phys.} \textbf{96} (5-6), 1331-1341 (1999).

\bibitem[MVY11]{MVY} T.\ Mountford, D.\ Valesin, Q.\ Yao. Metastable densities for contact processes on random graphs.  Preprint, arXiv:1106.4336v1 (2011).

\bibitem[NSW01]{NSW} M.E.J.\ Newman, S.H.\ Strogatz, D.J.\ Watts. Random graphs with arbitrary degree distributions and their applications. \emph{Phys. Rev. E} \textbf{64}, 026118 (2001).

\bibitem[Pe92]{P} R.\ Pemantle. The contact process on trees. \emph{Ann. Probab.} \textbf{20}, 2089-2116 (1992).

\bibitem[Sc85]{schonmeta} R.\ Schonmann. Metastability for the contact process. \emph{J. Statist. Phys.} \textbf{41} (3-4), 445-464 (1985).

\bibitem[St01]{St} A.\ Stacey. The contact process on finite homogeneous trees. \emph{Probab. Theory Related Fields} \textbf{121} (4), 551-576 (2001).


\end{thebibliography}
\end{document}